\documentclass[11pt, fleqn]{amsart}
\usepackage{amscd, amsfonts, amsthm, graphicx, amssymb}
\usepackage[all]{xy}

\newtheorem{theorem}{Theorem}[section]
\newtheorem{proposition}{Proposition}[section] 
\newtheorem{corollary}{Corollary}[section] 
\newtheorem{lemma}{Lemma}[section]

\theoremstyle{definition}

\newtheorem{defn}{Definition}[section]

\newtheorem{remark}{Remark}[section]

\newtheorem{example}{Example}[section]

\newcommand{\ds}{\displaystyle}
\newcommand{\bs}{\boldsymbol}

\newcommand{\cal}{\mathcal}
\renewcommand{\phi}{\varphi}

\newcommand{\lra}{\longrightarrow}
\newcommand{\LLRA}{\ $\Longleftrightarrow$ \ } 
\newcommand{\LRA}{\ $\Longrightarrow$ \ }

\newcommand{\e}{\varepsilon}
\newcommand{\id}{\mathrm{id}}
\newcommand{\LD}{\mbox{$({\rm LD})$}}
\newcommand{\ED}{\mbox{$({\rm ED})$}}

\newcommand{\E}{\mathcal E}

\newcommand{\U}{\mathcal U}

\newcommand{\W}{\mathcal W}

\newcommand{\IR}{\mathbb R}
\newcommand{\IH}{\mathbb H}
\newcommand{\IN}{\mathbb N}
\newcommand{\IZ}{\mathbb Z}

\newcommand{\IS}{\mathbb S}

\begin{document}
\baselineskip 6 mm

\thispagestyle{empty}


\title[Local and end deformation theorems for uniform embeddings]
{Local and end deformation theorems for uniform embeddings} 

\author[Tatsuhiko Yagasaki]{Tatsuhiko Yagasaki}
\address{Graduate School of Science and Technology, Kyoto Institute of Technology, Kyoto, 606-8585, Japan}
\email{yagasaki@kit.ac.jp}

\subjclass[2010]{57S05; 58D10, 57N15, 54E40} 
\keywords{Space of uniform embeddings, Group of uniform homeomorphisms, Deformation Theorem, Uniform isotopies}

\maketitle

\begin{abstract} 
A local deformation property for uniform embeddings in metric manifolds (LD) 
is formulated and its behaviour is studied in a formal view point. 
It is shown that any metric manifold with a geometric group action, 
typical metric spaces (Euclidean space, hyperbolic space and cylinders)  
and for $\kappa \leq 0$ the $\kappa$-cone ends over any compact Lipschitz metric manifolds,  
all of them have the property (LD).  
We also formulate a notion of end deformation property for uniform embeddings over proper product ends (ED). 
For example, the 0-cone end over a compact metric manifold has the property (ED) if it has the property (LD). 
It is shown that if a metric manifold $M$ has finitely many proper product ends with the property (ED), then 
the group of bounded uniform homeomorphisms of $M$ endowed with the uniform topology 
admits a strong deformation retraction onto 
the subgroup of bounded uniform homeomorphisms which are identity over those ends. 
We also study a role of uniform isotopies in deformation of uniform homeomorphisms and 
show that Alexander isotopies in $\kappa$-cones induce 
contractions of some subgroups of groups of bounded uniform homeomorphisms. 
\end{abstract} 


\section{Introduction} 

This article is a continuation of study of topological properties of spaces of uniform embeddings and 
groups of uniform homeomorphisms in metric manifolds \cite{Ce, Ki, MSYY, Ya}. 
Since the notion of uniform continuity and uniform topology depends on the choice of metrics, 
it is essential to select reasonable classes of metric manifolds $(M,d)$.   
J. M. Kister \cite{Ki} studied the case of the standard Euclidean space $\IR^n$ using Alexander's trick. 
{A.V.~{\v C}ernavski\u\i} \cite{Ce} considered the case where $M$ is the interior of a compact manifold $N$ and the metric $d$ is a restriction of some metric on $N$. 
Recently, in \cite{Ya} we considered the class of metric covering spaces over compact manifolds 
and obtained a local deformation theorem for uniform embeddings in those spaces. 
From this local deformation theorem we also deduced a global deformation result on groups of uniform homeomorphisms of metric manifolds with finitely many Euclidean ends. 
As an application, 
we showed that the group ${\cal H}_b^u(\IR^n)$ of bounded uniform homeomorphisms of $\IR^n$ endowed with the uniform topology 
is contractible as a topological space. 

In this point, it is important to notice that there are 
two approaches to discuss topological properties of various subgroups $G$ of homeomorphism groups of (compact or non-compact) manifolds. 
One way is to study ordinary topological properties as topological spaces under given topologies. 
Another way is to modify various notions on topological spaces based upon characteristic features of subgroups $G$ 
and study those modified properties. 
A typical example is the notion of paths in $G$. 
For groups of PL-homeomorphisms it is natural and useful to define a path as a PL-isotopy 
(independent of choice of topologies in groups of PL-homeomorphisms). 
Although those modified notions are usually described by the same terminologies as original ones in literatures, 
we should distinguish these two approaches rigorously and should not confuse them (as explained below). 

In the fundamental textbook \cite[Section 5.6]{Ru} a modification of notion of (local) contractibility is considered, 
since homeomorphism groups of non-compact manifolds are not necessarily locally contractible as topological spaces 
under any of the Whitney, uniform and compact-open topologies. 
In this textbook and related papers, a path is interpreted as an isotopy 
and the notion of (local) contractibility is reformulated as (local) existence of continuous selection of isotopies to the identity under a given topology.
Under this weaker notion of (local) contractibility, 
homeomorphism groups of any non-compact manifolds endowed with the Whitney topology is locally contractible \cite{Ce, EK, Ru} and 
the group ${\cal H}_b(\IR^n)$ of bounded homeomorphisms of $\IR^n$ endowed with the uniform topology is contractible by means of Alexander isotopy in $\IR^n$ (cf.\,\cite{Ki}). 
However, this weak notion of local contractibility does not imply even local connectedness of homeomorphism groups as topological spaces 
and the contraction of ${\cal H}_b(\IR^n)$ defined by Alexander isotopy is not continuous with respect to the uniform topology 
(Examples~\ref{exp_Alexander_isotopy_IR^n}, ~\ref{exp_Alexander_isotopy_kappa-cones}). 
These differences between two approaches might lead to some confusion in understanding of topological properties of homeomorphism groups of (non-compact) manifolds. 
Here we emphasize that, throughout \cite{MSYY, Ya} and this paper,  
we have studied the ordinary topological properties of groups of uniform homeomorphisms endowed with the uniform topology  as topological spaces. 
In comparison with classical results, 
this difference of approach yields a contrast between our statements and classical ones in some cases   
and also causes a situation 
that apparently same statements have different meanings in other cases. 

In this article we extract essence from the results obtained in \cite{Ya} and 
formulate notions of local deformation property over subsets and 
global deformation property over product ends for uniform embeddings in metric manifolds (see Sections 3 and 5 for the precise definitions). 
Below we abbreviate the phrase ``the local deformation property for uniform embeddings'' as (LD) and 
``the end deformation property for uniform embeddings'' as (ED). 

Study of basic behaviour of these deformation properties enables us to deduce 
deformation results in more complicated metric manifolds from those in simpler pieces. 
For example, the additivity of (LD) (Proposition~\ref{prop_LD}) can be applied to show that 
any metric manifold with a locally geometric group action has the property (LD) (Theorem~\ref{thm_LD_geom_action}).  
Here a group action on a metric manifold is called (locally) geometric if it is proper, cocompact and (locally) isometric (cf.~Section 4.1). 
In \cite{Ya} we considered metric covering spaces over compact manifolds. 
In terms of covering transformation groups, 
these correspond to metric manifolds with free locally geometric group actions. 
Therefore, this is regarded as a generalization from the free case to the non-free case. 
The key observation is that any metric manifold with a locally geometric group action is 
a finite union of invariant open subsets each of which is the total space of a trivial metric covering projection.  
Thus the non-free case follows from the free case and the additivity of (LD).

\begin{theorem}\label{thm_LD_geom_action_Intro} 
A metric manifold has the property \LD\,if it admits a locally geometric group action. 
\end{theorem} 

\newpage 

It is also shown that the property (LD) in a metric manifold depends only on the property (LD) of  its ends (Corollary~\ref{cor_LD}). 
This observation leads us to a study of deformation property for uniform embeddings over typical metric ends (cf.~Section 4.2). 
The Euclidean space $\IR^n$, the hyperbolic space $\IH^n$ and any cylinders have the property (LD), 
since they admit free geometric group actions. 
Hence all of their ends also have the property (LD). 
Since the $\kappa$-cone $C_\kappa({\Bbb S}^{n-1})$ of 
the unit $(n-1)$-sphere ${\Bbb S}^{n-1}$ with the spherical metric is canonically isometric to 
$\IR^n$ for $\kappa = 0$ and to $\IH^n$ for $\kappa = -1$ (cf.~\cite[p 59]{BH}), 
one can expect that the above results extend to a result on more general cone ends. 
The next result follows from the sphere case and the additivity of (LD) (Proposition~\ref{prop_k-cone_LD}).   

\begin{theorem}\label{thm_kapp-cone_LD} Suppose $(N,d)$ is a compact Lipschitz metric manifold. 
Then the $\kappa$-cone end $C_\kappa(N,d)_1$ has the property $\LD$ when {\rm (i)} $\kappa = 0$ or {\rm (ii)} $\kappa < 0$ and $\partial N = \emptyset$. 
\end{theorem}

Here, a Lipschitz metric $n$-manifold is a metric $n$-manifold (possibly with boundary) 
which admits an atlas in which each chart is a (locally) bi-Lipschitz equivalence onto an open subset of the half space $\IR^n_{\geq 0}$. 
For example, any Riemannian manifold is a Lipschitz metric manifold. 

Next we study basic behaviour of the property (ED). 
In \cite{Ya} we deduced the property (ED) in Euclidean ends from the property (LD). 
This extends to the next form (Example~\ref{exp_ED=>ED}). 

\begin{proposition}
Suppose $(N, d)$ is a compact metric manifold. 
Then the $0$-cone end $C_0(N, d)_1$ has the property $\ED$ if it has the property $\LD$. 
\end{proposition}

In particular, if $(N,d)$ is a compact Lipschitz metric manifold, then 
$C_0(N,d)_1$ has the property $\ED$.
In \cite{Ya} 
we showed that if $(M ,d)$ is a metric manifold with finitely many Euclidean ends, then 
the group ${\cal H}^u_b(M, d)$ of bounded uniform homeomorphisms of $(M, d)$ endowed with the uniform topology 
admits a strong deformation retraction onto 
the subgroup of ${\cal H}^u_b(M, d)$ consisting of bounded uniform homeomorphisms which are identity over those ends. 
This result generalizes to the next form. 

\begin{theorem}\label{thm_ED_sdr_H}
Suppose $(M, d)$ is a metric manifold and 
$L(1), \cdots, L(m) \subset M$ are pairwise disjoint isolated proper product ends of $M$ such that   
$(L(i), d|_{L(i)})$ has the property $\ED$ and 
$L(i)_2 \subset {\rm Int}_M L(i)$ for $i =1, \cdots, m$. 
Let $L_r = L(1)_r \cup \cdots \cup L(m)_r$ $(r \geq 1)$. 
Then there exists a strong deformation retraction 
$\phi$ of ${\cal H}^u_b(M, d)$ onto ${\cal H}^u_b(M, d; L_3)$ 
such that 
\begin{enumerate} 
\item for each $(h, t) \in {\cal H}^u_b(M, d) \times [0,1]$ \hspace{3mm} 
\begin{tabular}[t]{c@{\ \,}l}
{\rm (i)} & $\phi_t(h) = h$ \ on \ $h^{-1}(M - {\rm Int}_M L_2) - {\rm Int}_M L_2$, \\[2mm]  
{\rm (ii)} & if $h = \id$ on $\partial M$, then $\phi_t(h) = \id$ on $\partial M$, 
\end{tabular}
\vspace{1mm} 
\item for each $  \alpha > 0$ there exists $ \beta > 0$ such that  
$$\phi({\cal H}^u(\id_M, \alpha; (M, d)) \times [0,1]) \subset {\cal H}^u(\id_M, \beta; (M, d)).$$ 
\end{enumerate} 
\end{theorem}

See Sections 2 and 5 for the precise definitions of terminologies used in Theorem~\ref{thm_ED_sdr_H}. 
In Section 5 we obtain a more general end deformation theorem for uniform embeddings (Theorem~\ref{thm_ED_sdr}). 

We also study some properties of uniform isotopies from the viewpoint of deformation of uniform homeomorphisms. 
We can interpret (local) contractibility of 
the group ${\cal H}^u_b(X, d)$ of bounded uniform homeomorphisms of a metric space $(X,d)$ 
in terms of a continuous selection of {\em uniform} isotopies on $(X,d)$ (Corollary~\ref{cor_LC_uniform isotopy}).
As a typical example, we discuss basic properties of Alexander isotopies in the $\kappa$-cone $C_\kappa(X, d)$ $(\kappa \leq 0)$ over a compact metric space $(X,d)$. 
First we note that Alexander isotopies only provide a continuous selection of {\em levelwise bounded} isotopies on $C_\kappa(X, d)$, 
that is, a contraction of ${\cal H}_b(C_\kappa(X, d))$ in the sense of \cite{Ru}.  
So, it is worth investigating  the proper subset ${\cal G}$ 
consisting of $h \in {\cal H}(C_\kappa(X, d))$ such that the Alexander isotopy of $h$ is a uniform isotopy on $C_\kappa(X, d)$. 
It is shown that ${\cal G}$ is a maximal subgroup of ${\cal H}^u_b(C_\kappa(X, d))$ on which Alexander isotopies induce a contraction in the usual sense 
and that ${\cal G}$ includes the subgroup ${\cal H}_0(C_\kappa(X, d))$ of homeomorphisms asymptotic to $\id$ at $\infty$ 
(Propositions~\ref{prop_G}, ~\ref{prop_contraction_H_0}). 

This paper is organized as follows. 
Section 2 includes basics on metric spaces and uniform embeddings. 
In Section 3 we study formal behaviour of the property (LD) with respect to restriction, union, ends and uniform equivalence. 
In Sections 4 we discuss some examples of metric manifolds with the property (LD), including 
metric manifolds with locally geometric group actions (Section 4.1), the ends of Euclidean space, hyperbolic space and cylinders (Section 4.2) and the $\kappa$-cone ends over compact Lipschitz metric manifolds (Section 4.3).
Section 5 is devoted to a study of basic behaviour of the property (ED) and the proof of Theorems~\ref{thm_ED_sdr_H} and~\ref{thm_ED_sdr}. 
Sections 6 and 7 include results on uniform isotopies and Alexander isotopies. 
 

\section{Metric spaces and uniform embeddings} 

This section includes basics on metric spaces and uniform embeddings needed in this paper. 

\subsection{Conventions} \mbox{} 

In this article, maps between topological spaces are always assumed to be continuous. 
For a topological space $X$ and a subset $A$ of $X$, 
the symbols ${\rm Int}_X A$,  $cl_X A$ and ${\rm Fr}_X A$ denote the topological 
interior, closure and frontier of $A$ in $X$. 
Let ${\cal K}(X)$ denote the collection of all compact subsets of $X$. 

As usual, any metric space $(X, d)$ is given the topology $\tau_d$ induced by the metric $d$, and 
a metric on a topological space $(Y, \tau)$ means a metric $d$ on the set $Y$ with $\tau_d = \tau$. 
(Frequently, the symbols $d$ and $\tau_d$ are omitted when they are understood implicitly.)  
A metric space is said to be proper if any bounded closed subset is compact. 

Suppose $(X, d)$ is a metric space. 
The metric of any subset $A$ of $X$ is given by the restriction $d|_A$ of the metric $d$ to $A$. 
The distance between two subsets $A$ and $B$ of $X$ is defined by $d(A, B) = \inf \{ d(a,b) \mid a \in A, b \in B \}.$ 
For $x \in X$ and $\e \in [0, \infty]$ let $O_\e(x)$ and $B_\e(x)$ denote the open and closed $\e$-neighborhood of $x$ in $(X,d)$ respectively 
(i.e., $O_\e(x) = \{ y \in X \mid d(x, y) < \e \}$, $B_\e(x) = \{ y \in X \mid d(x, y) \leq \e \}$, $O_\infty(x) = X$). 
The open $\e$-neighborhood of $A$ in $(X,d)$ is defined by 
$$O_\e(A) = O_\e(A, X) = \{ x \in X \mid d(x,A) < \e \}.$$ 
Note that (i) (a) $d(x,A) < \e$ if and only if $d(x, a) < \e$ for some $a \in A$, 
(b) $O_\e(O_\delta(A)) \subset O_{\e+\delta}(A)$ and 
(ii) for $A \subset B \subset X$ (a) $O_\e(A) \subset B$ $\Longleftrightarrow$ $d(A, X - B) \geq \e$ 
$\Longleftrightarrow$ $O_\e(X - B) \subset X - A$, 
(b) $O_\e(A) \subset O_\e(B)$ and $O_\e(A, B) = O_\e(A) \cap B$.   
We write $A \subset_u B$ in $X$ and call $B$ a uniform neighborhood of $A$ in $X$ if $O_\e(A) \subset B$ for some $\e > 0$.  
Note that if $A \subset_u B$ in $X$, then $A \cap Y \subset_u B \cap Y$ in $Y$ for any $Y \subset X$. 

We say that a subset $A$ of $X$ is $\e$-discrete if $d(x,y) \geq \e$ for any two distinct points $x, y \in A$ and 
that $A$ is uniformly discrete if it is $\e$-discrete for some $\e > 0$. 
More generally a family $\{ F_\lambda \}_{\lambda \in \Lambda}$ of subsets of $X$ is said to be $\e$-discrete 
if $d(F_\lambda, F_\mu) \geq \e$ for any two distinct $\lambda, \mu \in \Lambda$. 

\subsection{Maps between metric spaces} \mbox{} 

Suppose $(X,d)$ and $(Y,\rho)$ are metirc spaces. 
A map $f : (X,d) \to (Y,\rho)$ is said to be 
(i) uniformly continuous if for each $\e > 0$ there is a $\delta > 0$ such that 
if $x,x' \in X$ and $d(x,x') < \delta$ then $\rho(f(x), f(x')) < \e$, 
(ii) coarsely uniform if for each $R > 0$ there is a $S > 0$ such that 
if $x,x' \in X$ and $d(x,x') < R$ then $\rho(f(x), f(x')) < S$, 
(iii) Lipschitz if there exists $K > 0$ such that $\rho(f(x), f(x')) \leq Kd(x,x')$ for any $x, x' \in X$. 
The constant $K$ is called a Lipschitz constant for $f$. 
Note that 
(a) if $A \subset_u U \subset X$ and both $f|_U$ and $f|_{X - A}$ are uniformly continuous, then $f$ is also uniformly continuous, 
(b) if $f$ is uniformly continuous and $A \subset_u B \subset Y$, 
then $f^{-1}(A) \subset_u f^{-1}(B) \subset X$ and 
(c) if $f$ is Lipschitz, then $f$ is both uniformly continuous and coarsely uniform. 

A map $f : (X,d) \to (Y,\rho)$  is called a uniform (coarsely uniform, Lipschitz) homeomorphism 
if $f$ is bijective and both $f$ and $f^{-1}$ are uniformly continuous (coarsely uniform, Lipschitz, respectively). 
(The prefix ``bi'' is sometimes added to emphasize that $f^{-1}$ also satisfies the same condition.)
A uniform embedding is a uniform homeomorphism onto its image. 
A map $h : (X,d) \to (Y,\rho)$ is called a similarity transformation if $h$ is bijective and there exists $K > 0$ such that $\rho(h(x), h(x')) = Kd(x,x')$ for any $x, x' \in X$. 

In \cite{Ya} we introduced the notion of metric covering projections  
as a natural metric version of Riemannian coverings in the smooth category. 
For the basics on covering spaces, we refer to \cite[Chapter 2, Section 1]{Sp}.

\begin{defn} A map $\pi : (X,d) \to (Y, \rho)$ is called a metric covering projection 
if it satisfies the following conditions: 
\begin{itemize} 
\item[$(\ast)_1$] There exists an open cover ${\cal U}$ of $Y$ such that for each $U \in {\cal U}$ the inverse $\pi^{-1}(U)$ is the disjoint union of open subsets of $X$ each of which is mapped isometrically onto $U$ by $\pi$. 
\item[$(\ast)_2$] For each $y \in Y$ the fiber $\pi^{-1}(y)$ is uniformly discrete in $X$. 
\item[$(\ast)_3$] $\rho(\pi(x), \pi(x')) \leq d(x, x')$ for any $x, x' \in X$. 
\end{itemize}
\end{defn}

\subsection{Metric manifolds} \mbox{} 

An $n$-manifold means a topological $n$-manifold possibly with boundary. 
For a manifold $M$, the symbols $\partial M$ and ${\rm Int}\,M$ denote the boundary and interior of $M$ as a manifold.  

A metric $n$-manifold means an $n$-manifold with a fixed compatible metric. 
A basic model is the simply connected complete Riemannian $n$-manifold 
$M_\kappa^n$ with constant sectional curvature $\kappa$. 
For example, $M_\kappa^n = \IS^n, \IR^n, \IH^n$ for $\kappa = 1, 0, -1$ respectively. 
Here, $\IR^n$ is the Euclidean $n$-space with the standard Euclidean metric $d_0$, $\IH^n$ is the hyperbolic $n$-space and 
${\Bbb S}^n = \{ x \in \IR^{n+1} \mid \| x \| = 1 \}$ is the standard unit $n$-sphere in $\IR^{n+1}$ 
with the standard spherical metric $d_1$ defined by 
$\cos d_1(x, y) = \langle x, y\rangle$, where $\langle \ , \ \rangle$ is the standard Euclidean inner product. 
Note that $M_{\pm \kappa}^n$ is homothetic to $M_{\pm 1}^n$ for any $\kappa > 0$. 
Let $\IR^n_{\geq 0} = \{ x \in \IR^n \mid x_n \geq 0\}$ denote the closed half space in $\IR^n$ with the standard metric $d_0$. 

\begin{defn} A Lipschitz metric $n$-manifold is a metric $n$-manifold $(M, d)$ which satisfies the following condition: 
\begin{itemize}

\item[$(\natural)$] For each point $x \in M$ there exists an open neighborhood $U$ of $x$ in $M$ and an open subset $V$ of $\IR^n_{\geq 0}$ such that 
$(U,d|_U)$ is bi-Lipschitz homeomorphic to $(V, d_0|_V)$. 
\end{itemize}
\end{defn} 

Any Riemannian manifold with the induced path-length metric is a Lipschitz metric manifold. 


\subsection{Spaces of uniform embeddings} \mbox{} 

Suppose $(X,d)$ and $(Y,\rho)$ are metirc spaces. 
Let ${\cal C}((X,d), (Y,\rho))$ denote the space of continuous maps $f : (X,d) \to (Y,\rho)$.  
The metric $\rho$ on $Y$ induces the sup-metric $\tilde{\rho}$ on ${\cal C}(X, Y)$ defined by 
$$\tilde{\rho}(f,g) = \sup \{ \rho(f(x), g(x)) \mid x \in X \} \in [0, \infty].$$ 
The topology on ${\cal C}((X,d), (Y,\rho))$ induced by this sup-metric $\tilde{\rho}$ is called the uniform topology. 
Below the space ${\cal C}((X,d), (Y,\rho))$ and its subspaces are endowed with the sup-metric $\tilde{\rho}$ and 
the uniform topology, unless otherwise specified. 
(When we need to emphasize this point, we add the subscript $u$ as ${\cal C}((X,d), (Y,\rho))_u$.) 
For notational simplicity, the metric $\tilde{\rho}$ is denoted by the same symbol $\rho$. 
Let ${\cal C}^u((X,d), (Y,\rho))$ denote the subspace consisting of uniformly continuous maps $f : (X,d) \to (Y,\rho)$.  
For any metric spaces $(X_1,d_1)$, $(X_2,d_2)$, $(X_3,d_3)$ the following map is continuous : 
$${\cal C}^u((X_1,d_1), (X_2,d_2)) \times {\cal C}^u((X_2,d_2), (X_3,d_3)) \lra {\cal C}^u((X_1,d_1), (X_3,d_3)) : (f, g) \longmapsto gf.$$ 

\begin{lemma}\label{lemma_conti} $($\cite[Lemma 2.5]{Ya}$)$ Suppose $P$ is a topological space, 
$f : P \to {\cal C}(X_1, X_2)$, $g : P \to {\cal C}(X_1, X_3)$ are continuous maps 
and $h : P \to {\cal C}^u(X_2, X_3)$ is a function. 
If $f_p$ is surjective and $h_pf_p = g_p$ for each $p \in P$, then $h$ is continuous. 
\end{lemma}

Let ${\cal H}(X,d)$ denote the group of homeomorphisms $h$ of $(X,d)$ onto itself.  
This group and its subgroups are endowed with the sup-metric and the uniform topology 
(as the subspaces of ${\cal C}((X,d), (X,d))$). 
Let ${\cal H}_b(X,d) := \{ h \in {\cal H}(X,d) \mid d(h, \id_X) < \infty \}$ (the subgroup of bounded homeomorphisms). 

Let ${\cal H}^u(X,d)$ denote the subgroup of uniform homeomorphisms of $(X,d)$ onto itself.  
It is a topological group, 
while the whole group ${\cal H}(X,d)$ is not necessarily a topological group under the uniform topology. 
We also use the following notations: 
$$\mbox{${\cal H}^u_b(X,d) := {\cal H}^u(X,d) \cap {\cal H}_b(X,d)$, \ \ 
${\cal H}^u(X,d; C) := \{ h \in {\cal H}^u(X,d) \mid h = \id \text{ on } C\}$ \ for $C \subset X$}$$  
and for $\alpha > 0$ let ${\cal H}^u(\id_X, \alpha; X, d)$ denote the open $\alpha$-neighborhood of $\id_X$ in ${\cal H}^u(X, d)$. 
The group ${\cal H}^u_b(X,d)$ is an open (and closed) subgroup of ${\cal H}^u(X,d)$, so 
it includes the connected component of $\id_X$ in ${\cal H}^u(X,d)$. 
The group ${\cal H}^u_b(X,d)$ also includes the subgroups ${\cal H}_c(X,d) \subset {\cal H}_0(X,d)$, where 
\begin{itemize}
\item[(a)] ${\cal H}_0(X, d) = \{ h \in {\cal H}(X, d) \mid h : (\ast) \}$ \ (the subgroup of homeomorphisms asymptotic to $\id$): \\ 
\hspace*{3mm} $(\ast)$ : for any $\e > 0$ there exists $K \in {\cal K}(X,d)$ such that $d(h(x), x) < \e$ $(x \in X - K)$, 
\item[(b)] ${\cal H}_c(X, d) = \{ h \in {\cal H}(X, d) \mid {\rm supp}\,h \in {\cal K}(X,d)\}$ \ \ (${\rm supp}\,h = cl_X \{ x \in X \mid h(x) \neq x \}$). 
\end{itemize} 
Note that ${\cal H}_0(X, d)$ is a normal subgroup of ${\cal H}^u(X, d)$ 
and ${\cal H}_c(X, d)$ is a normal subgroup of ${\cal H}(X, d)$. 

For subsets $F, C \subset X$,  
let $\E^u(F, X, d; C)$ denote the space of uniform embeddings 
$$\mbox{$f : (F,d|_F) \to (X, d)$ \ \ with \ $f = \id$ \ on $F \cap C$}$$    
with the sup-metric and the uniform topology. 

Suppose $(M,d)$ is a metric $n$-manifold and $F, C \subset M$. 
An embedding $f : F \to M$ is said to be proper if $f^{-1}(\partial M) = F \cap \partial M$. 
Let ${\mathcal E}_\ast^u(F, (M, d); C)$ 
denote the subspace of 
${\mathcal E}^u(F, (M, d); C)$ consisting of proper uniform embeddings 
and for $\alpha > 0$ let $\E^u_\ast(i_F, \alpha; F, (M, d); C)$ denote 
the open $\alpha$-neighborhood of the inclusion map $i_F : F \subset M$ in the space $\E^u_\ast(F, (M, d); C)$. 
Let $\E^u_\ast(F, (M, d); C)_b = \{ f \in \E^u_\ast(F, (M, d); C) \mid d(f, i_F) < \infty\}$. 

By deformations of uniform embeddings we mean the following type of homotopies in spaces of uniform embeddings: 

\begin{defn}\label{def_adm_def}
Suppose ${\cal W}$ is a neighborhood of $i_F$ in $\E^u_\ast(F, (M, d); C)$ and $H \subset F$, $D \subset C$. 
An admissible deformation of ${\cal W}$ over $H$ in $\E^u_\ast(F, (M, d); D)$ means a homotopy 
$$\phi : {\cal W} \times [0,1] \lra \E^u_\ast(F, (M, d); D)$$   
which satisfies the following conditions:  
\begin{enumerate} 
\item For each $f \in {\cal W}$ \ \ (i) $\phi_0(f) = f$, \ \ (ii) $\phi_1(f) = \id$ \ on $H$,  
\begin{itemize}
\item[] (iii) if $f = \id$ on $F \cap \partial M$, then $\phi_t(f) = \id$ on $F \cap \partial M$ \ $(t \in [0,1])$.  
\end{itemize} 
\item $\phi_t(i_F) = i_F$ \ $(t \in [0,1])$. 
\end{enumerate} 
\end{defn}

The adjectives ``local'' and ``global'' to deformations represent  
the size of the neighborhood ${\cal W}$ in the ambient space $\E^u_\ast(F, (M, d); C)$ 
(in other words, the size of $\alpha$ with $\E^u_\ast(i_F, \alpha; F, (M, d); C) \subset {\cal W}$). 

\subsection{$\kappa$-cones} \mbox{} 

First we recall the basics on $\kappa$-cones over metric spaces (cf. \cite[Chapter I.5]{BH}). 
The cone $C(Y)$ over a set $Y$ is defined by   
$C(Y) = \big(Y \times [0, \infty)\big)/\big(Y \times \{ 0 \}\big)$, that is, 
it is obtained from the set $Y \times [0, \infty)$ by collapsing the subset $Y \times \{ 0 \}$ to one point. 
In other words, $C(Y)$ is the quotient set of the set $Y \times [0, \infty)$ under the equivalence relation $\sim$ defined by \\[1mm] 
\hspace*{20mm} $(x, s) \sim (y,t)$ if and only if (a) $s = t = 0$ or (b) $s = t > 0$ and $x = y$. \\[1mm]
The equivalence class of the point $(x,s)$ is denoted by $sx$. 
The distinguished point $\bs{0} := 0x$ is called the cone point of $C(Y)$. 
For $sx \in C(Y)$ and $t \in [0, \infty)$ the scalor multiplication is defined by $t(sx) = (ts)x$. 
Let \\[1mm]
\hspace*{20mm} $C(Y)_+ := C(Y) - \{ \bs{0} \}$ \ \ and \ \ $C(Y)_r := \{ ty \in C(Y) \mid t \geq r \}$ $(r \geq 0)$. \\
Any function $f : Y \to Y'$ induces the associated function \\
\hspace*{40mm} $C(f) : C(Y) \to C(Y')$ : $C(f)(ty) = tf(y)$.  

Suppose $(Y, d)$ is a metric space. The metric $d_\pi$ on $Y$ is define by 
$$d_\pi(x, y) := \min \{ d(x, y), \pi \} \ \ (x, y \in Y).$$

\begin{defn} For $\kappa \leq 0$, 
the $\kappa$-cone $C_\kappa(Y, d)$ over $(Y,d)$ is the metric space $(C(Y), \tilde{d}_\kappa)$, 
where the metric $\tilde{d}_\kappa$ on $C(Y)$ is defined by the following formula: 
\begin{itemize}
\item[(i)\,] for $\kappa = 0$ : \hspace{5mm} $\tilde{d}_\kappa(sx , ty)^2 = s^2 + t^2 - 2st \cos d_\pi(x,y)$,  
\vskip 1mm 
\item[(ii)] for $\kappa < 0$ : \hspace{5mm} 
$\begin{array}[t]{@{}l}
\cosh (\sqrt{-\kappa} \, \tilde{d}_\kappa (sx, ty)) \\[2mm] 
\hspace*{10mm} = 
\cosh (\sqrt{-\kappa} \,  s) \cosh (\sqrt{-\kappa} \,  t) - \sinh (\sqrt{-\kappa} \,  s) \sinh (\sqrt{-\kappa} \,  t)  \cos d_\pi(x, y). \end{array}$ 
\end{itemize}
\vskip 1mm 
The $\kappa$-cone end over $(Y,d)$ means  the subspaces $C_\kappa(Y,d)_r = (C(Y)_r, \tilde{d}_\kappa)$ $(r> 0)$ of 
the $\kappa$-cone $C_\kappa(Y,d)$. Let $C_\kappa(Y,d)_+ := (C(Y)_+, \tilde{d}_\kappa)$. 
\end{defn} 

\begin{example}\label{exp_cone_sph}(cf.~\cite[Proposition 5.8]{BH}) 
Consider the standard unit $(n-1)$-sphere ${\Bbb S}^{n-1}$ in $\IR^n$ with the standard spherical metric $d_1$. 
Then, the $\kappa$-cone $C_\kappa({\Bbb S}^{n-1})$ is isometric to $M_\kappa^n$ for $\kappa \leq 0$. 
An isometry $\phi : C_\kappa({\Bbb S}^{n-1}) \approx M_\kappa^n$ is obtained by 
$$\mbox{$\phi(ty) = \exp_o(ty)$ \hspace{5mm} $(ty \in C_\kappa({\Bbb S}^{n-1}))$,}$$ 
where $o$ is any fixed point of $M_\kappa^n$, $\exp_o : T_o M_\kappa^n \to M_\kappa^n$ is the exponential map and 
${\Bbb S}^{n-1}$ is identified to the unit sphere $S(T_o M_\kappa^n)$ in the tangent space $T_o M_\kappa^n$.
Consider the closed half space $\IR^n_{\geq 0} = \{ x \in \IR^n \mid x_n \geq 0 \}$ 
and the closed upper half $(n-1)$-sphere ${\Bbb S}^{n-1}_{\geq 0} = \{ x \in {\Bbb S}^{n-1} \mid x_n \geq 0 \}$. For $\kappa = 0$, with taking $o$ the origin of $\IR^n$, we have an isometry of pairs, 
$\phi : (C_0({\Bbb S}^{n-1}), C_0({\Bbb S}^{n-1}_{\geq 0})) \approx (\IR^n, \IR^n_{\geq 0})$. 
\end{example} 

\begin{remark}\label{rem_kappa-cone} \mbox{} 
\vspace*{-4mm} 
\begin{enumerate}
\item Consider the functions \ $\lambda_\kappa : \IR \to \IR$ \ $(\kappa \leq 0)$ \ defined by \ \  
$\ds \lambda_\kappa(u) = 
\left\{ \begin{array}[c]{ll}
\frac{u}{\,2\,} & (\kappa = 0) \\[3mm] 
\sinh \big(\frac{\sqrt{-\kappa}}{2}u\big) & (\kappa < 0). 
\end{array} \right.$ 
\vskip 2mm 
\begin{itemize}
\item[(i)\,] $\lambda_\kappa$ is a diffeomorphism which is a monotonically increasing odd function (so $\lambda_\kappa(0) = 0$). 
\item[(ii)] $K \lambda_\kappa(u) \leq \lambda_\kappa(Ku)$ for $K \geq 1$ and $u \geq 0$. 
\end{itemize}
\vskip 2mm 
We use the abbreviation: $\lambda_\kappa^2(u) = (\lambda_\kappa(u))^2$. 
\vskip 2mm 
\item 
Using the function $\lambda_\kappa$, we have the following equalities for the metric $\tilde{d}_\kappa$ $(\kappa \leq 0)$: 
\vskip 2mm 
\begin{itemize}
\item[(i)\,] $\lambda_\kappa^2  (\tilde{d}_\kappa (sx, ty)) 
= \lambda_\kappa^2 (s-t) + \lambda_\kappa(2s) \lambda_\kappa(2t) \sin^2 \frac{1}{\,2\,} d_\pi(x, y)$ \hspace{5mm}
$(sx, ty \in C(Y)).$
\vskip 2mm 

\item[(ii)] (a) $\tilde{d}_\kappa(sy, ty) = |s-t| \leq \tilde{d}_\kappa(sx, ty)$ \hspace{1mm} and \hspace{1mm} 
(b) $\lambda_\kappa (\tilde{d}_\kappa (tx, ty)) = \lambda_\kappa(2t) \sin \frac{1}{\,2\,} d_\pi(x, y)$.  
\end{itemize}
\vskip 2mm 

\item If $u, v \in [0,\pi/2]$, $K \geq 1$ and $v \leq Ku$, then $\sin v \leq K \sin u$. 
\end{enumerate}
\end{remark}

\begin{remark}\label{rem_kappa-cone_proj} (1) 
The projection $q : (Y,\tau_d) \times [0, \infty) \to C_\kappa(Y,d)$ is continuous and 
its restriction $q : Y \times (0, \infty) \to C_\kappa(Y,d)_+$ is a homeomorphism.  
When $(Y, d)$ is compact, the map $q : Y \times [0, \infty) \to C_\kappa(Y,d)$ is a quotient map so that 
the topology on $C(Y)$ induced from the metric $\tilde{d}_\kappa$ coincides with the quotient topology induced from $Y \times [0, \infty)$. 

(2) The scalor multiplication $C_\kappa(Y, d) \times [0,\infty) \ni (u, t) \mapsto tu \in C_\kappa(Y, d)$ is continuous. 
\end{remark}

Suppose $f : (Y,d) \to (Y', d')$ is a map between metric spaces.  
Then, for any $\kappa \leq 0$ the $\kappa$-cone extension 
$C_\kappa(f) \equiv C(f) : C_\kappa(Y, d) \to C_\kappa(Y', d')$ is also continuous. 
If $K \geq 1$ and $f : (Y, d) \to (Y', d')$ is $K$-Lipschitz, then so is the map $f_\pi \equiv f : (Y, d_\pi) \to (Y', d'_\pi)$. 

\begin{lemma}\label{lem_Lip} The following conditions are equivalent. 
\begin{enumerate}
\item $f_\pi$ is Lipschitz. \hspace{10mm}
{\rm (2)} $C_\kappa(f)$ is Lipschitz. \hspace{10mm}
{\rm (3)} $C_\kappa(f)$ is uniformly continuous. 
\end{enumerate}
\end{lemma}

\begin{proof} 
(1) $\Rightarrow$ (2): The map $f_\pi$ is $K$-Lipschitz for some $K \geq 1$, that is, 
$$\mbox{$d'_\pi(f(x), f(y)) \leq K d_\pi(x, y)$ \ $(x, y \in Y)$.}$$

For any $sx, ty \in C(Y)$ with $s \leq t$, from Remark~\ref{rem_kappa-cone} it follows that 
\begin{enumerate}
\item[(i)\ ] $\tilde{d'}_\kappa(s f(y), t f(y)) = |s-t| \leq \tilde{d}_\kappa(sx, ty)$, \hspace{4mm} 
(ii) \ $\tilde{d}_\kappa(sx, sy) \leq \tilde{d}_\kappa(sx, ty)$ \ \ and 
\item[(iii)] $\tilde{d'}_\kappa(s f(x), s f(y)) \leq K \tilde{d}_\kappa(sx, sy)$, since 
$$\hspace{10mm} 
\begin{array}[t]{@{}l@{\ }l@{\ }l}
\ds \lambda_\kappa (\tilde{d'}_\kappa (sf(x), sf(y))) 
&=& \lambda_\kappa(2s) \sin \frac{1}{\,2\,} d'_\pi(f(x), f(y))\\[3mm] 
&\leq & \lambda_\kappa(2s) K \sin \frac{1}{\,2\,} d_\pi(x, y)
= K \lambda_\kappa (\tilde{d}_\kappa (sx, sy)) 
\leq \lambda_\kappa (K \tilde{d}_\kappa (sx, sy)).
\end{array}$$ 
\end{enumerate}
Therefore, we have 
\begin{eqnarray*}
\tilde{d'}_\kappa(s f(x), t f(y))  
&\leq & \tilde{d'}_\kappa(s f(x), s f(y)) + \tilde{d'}_\kappa(s f(y), t f(y)) \\
&\leq & K \tilde{d}_\kappa(sx, sy) + \tilde{d}_\kappa(sx, ty) \ \leq \ (K+1) \tilde{d}_\kappa(sx, ty).
\end{eqnarray*}

(3) $\Rightarrow$ (1): 
Conversely, suppose $f : (Y, d_\pi) \to (Y', d'_\pi)$ is not Lipschitz. 
Since $C_\kappa(f)$ is uniformly continuous, for $\e = 1$ there exists $\delta > 0$ such that 
$$\mbox{
$\tilde{d'}_\kappa(sf(x), tf(y)) < 1$ \ for any $sx, ty \in C_\kappa(Y, d)$ \ with \ $\tilde{d}_\kappa(sx, ty) \leq \delta$.
}$$  
Let $K := \frac{\lambda_\kappa(1)}{\lambda_\kappa(\delta)} > 0$. 
Since ${\ds \lim_{\mbox{\tiny $a \to 0$}}} \frac{\sin 2Ka}{\sin a} = 2K$, \vspace{1mm} 
there exists $a_0 \in (0, 1]$ such that $\frac{\sin 2Ka}{\sin a} > K$ for any $a \in (0, a_0]$. 

Take $L > 2K$ such that $\frac{\pi}{L} < a_0$. 
Since $f : (Y, d_\pi) \to (Y', d'_\pi)$ is not $L$-Lipschitz, 
there exist $x, y \in Y$ such that $d'_\pi(f(x), f(y)) > L d_\pi(x, y)$.  
Since $d'_\pi(f(x), f(y)) \leq \pi$, for $a := \frac{1}{\,2\,} d_\pi(x, y)$, it follows that 
$0 < a < d_\pi(x, y) \leq \frac{\pi}{L} \leq a_0$ \ and 
$$\sin \mbox{$\frac{1}{\,2\,}$} d'_\pi(f(x), f(y))
\geq \sin La 
> \sin 2K a 
> K \sin a= K \sin \mbox{$\frac{1}{\,2\,}$} d_\pi(x, y).$$

Take $t > 0$ such that $\lambda_\kappa(2t)  \sin \frac{1}{\,2\,} d_\pi(x, y) = \lambda_\kappa(\delta)$. 
Then $\tilde{d}_\kappa (tx, ty) = \delta$, but $\tilde{d'}_\kappa(tf(x), tf(y)) > 1$ since 
$$\begin{array}[t]{l}
\lambda_\kappa (\tilde{d}_\kappa (tx, ty)) 
= \lambda_\kappa(2t) \sin \mbox{$\frac{1}{\,2\,}$} d_\pi(x, y) = \lambda_\kappa(\delta) \hspace{5mm} \text{and} \\[3.5mm]  
\begin{array}[t]{@{}l@{\ }l@{\ }l}
\lambda_\kappa (\tilde{d'}_\kappa (tf(x), tf(y))) 
&=& \lambda_\kappa(2t) \sin \mbox{$\frac{1}{\,2\,}$} d'_\pi(f(x), f(y)) \\[2mm] 
&>&  K \lambda_\kappa(2t) \sin \mbox{$\frac{1}{\,2\,}$} d_\pi(x, y) 
= K \lambda_\kappa(\delta) = \lambda_\kappa(1). 
\end{array}
\end{array}$$
This contradiction completes the proof.  
\end{proof}


\section{Local deformation property for uniform embeddings} 

In this section we formulate a notion of local deformation property for uniform embeddings (LD) and study its basic behaviour. 
Throughout this section $(M, d)$ and $(M', d')$ denote metric manifolds. 

\subsection{Definition and basic properties} \mbox{}

\begin{defn}\label{def_adm_def_2} \mbox{} 
\begin{enumerate} 
\item An admissible tuple in $(M, d)$ is a tuple $(X, W', W, Z, Y)$ such that 
$X \subset_u W' \subset W \subset M$ and $Z \subset_u Y \subset M$. 
Let ${\cal S}(M, d)$ denote the collection of all admissible tuples in $(M, d)$. 
\item An admissible deformation for $(X, W', W, Z, Y) \in {\cal S}(M, d)$ is 
an admissible deformation 
$$\phi : \W \times [0,1] \lra \E^u_\ast(W, M, d; Z)$$ 
of a neighborhood $\W$ of $i_W$ in $\E^u_\ast(W, M, d; Y)$ over $X$ \ (in the sense of Definition~\ref{def_adm_def}) \ 
such that \ for each $(f,t) \in \W \times [0,1]$ 
\vskip 1mm 
\begin{itemize} 
\item[(i)] $\phi_t(f) = f$ \ on \ $W - W'$ \ \ and \ \ (ii) $\phi_t(f)(W) = f(W)$. 
\end{itemize} 
\end{enumerate} 
\end{defn}
 
\begin{defn}\label{notation_LD} (LD) 
\begin{itemize} 
\item[(1)] For ${\cal X}\in {\cal S}(M, d)$ \hspace{2mm} 
(LD)$_{(M,d)}({\cal X})$ \LLRA There exists an admissible deformation for ${\cal X}$.
\item[(2)] $(M,d)$ : $\LD$ \LLRA (LD)$_{(M,d)}({\cal X})$ for any ${\cal X} \in {\cal S}(M, d)$
\item[(3)] For $A \subset M$  
\hspace*{3mm} $A$ : (LD)$_{(M,d)}$ (or $A$ : (LD) in $(M, d)$) \\
\hspace*{10mm} \LLRA  (LD)$_{(M,d)}({\cal X})$ for any ${\cal X} =  (X, W', W, Z, Y) \in {\cal S}(M, d)$ with $X \subset A$ 
\end{itemize}
\end{defn}

We also use the following additional notations. 
\begin{itemize} 
\item[(4)] For $X \subset M$ \\
\hspace*{1mm} $\LD_{(M,d)}(X)$ \LLRA $\LD_{(M,d)}({\cal X})$ for any ${\cal X} = (X, W', W, Z, Y) \in {\cal S}(M, d)$
\end{itemize} 
\vskip 2mm 

Note that  
(i) $(M,d)$ : $\LD$ \LLRA $M$ : $\LD_{(M,d)}$ and 
(ii) $A$ : (LD)$_{(M,d)}$ \LLRA$\LD_{(M,d)}(X)$ for any $X \subset A$.  

\begin{example}\label{exp_EK_covering} \mbox{} 
\begin{itemize}
\item[(1)] The local deformation theorem by  A.V.~{\v C}ernavski\u\i\ (\cite{Ce}) and R.\,D.~Edwards - R.\,C.~Kirby (\cite[Theorem 5.1]{EK}) is restated as follows: 
Any relatively compact subset $K$ of $M$ has the property $\LD_{(M,d)}$. 

\item[(2)] A metric manifold $(M, d)$ has the property (LD) if it admits a metric covering projection
$\pi : (M, d) \to (N, \rho)$ onto a compact metric manifold $(N, \rho)$ (\cite[Theorem 1.1]{Ya}). 
\end{itemize}
\end{example} 

The condition $\LD$ has the following basic properties. 

\begin{proposition}\label{prop_LD} \mbox{} 
\begin{itemize}
\item[(1)] Suppose $h : (M,d) \approx (M', d')$ is a uniform homeomorphism. Then {\rm (i)} for any $A \subset M$ \\ 
\hspace*{5mm} $A : \LD_{(M,d)}$ \LLRA $h(A) : \LD_{(M', d')}$ \ and \ {\rm (ii)} $(M,d) : \LD$ \LLRA $(M', d') : \LD$

\item[(2)] 
\begin{itemize}
\item[(i)\,] Suppose $A \subset B \subset M$. Then \ $B: \LD_{(M,d)}$ \LRA $A : \LD_{(M,d)}$. 

\item[(ii)] Suppose $A \subset_u N \subset M$ and $N$ is an $n$-manifold.  
Then, \\
\hspace*{10mm} $A : \LD_{(N, d|_N)}$ \LLRA $A : \LD_{(M,d)}$. 
\end{itemize}
\vskip 1mm 

\item[(3)]
\begin{itemize}
\item[(i)]  Suppose $A \subset_u U \subset  M$ and $B \subset M$. Then \ $U$, $B : \LD_{(M,d)}$ \LRA $A \cup B : \LD_{(M,d)}$. 
\item[(ii)] If $A_1, \cdots, A_m \subset M$ and 
each $A_i$ admits a uniform neighborhood which satisfies $\LD_{(M,d)}$, then $\bigcup_{i=1}^m A_i$ satisfies $\LD_{(M,d)}$. 
\end{itemize}
\vskip 1mm 

\item[(4)] Suppose $K$ is a relatively compact subset of $M$. Then, for any $A \subset M$ \\
\hspace*{10mm} $A : \LD_{(M,d)}$ \LLRA $A \cup K : \LD_{(M,d)}$. 
\vskip 1mm 

\item[(5)]
If $(M,d) : \LD$ \ $($or more generally $\LD_{(M,d)}(M, M, M, \emptyset, \emptyset)$ holds$)$, then ${\cal H}^u(M, d)$ and ${\cal H}^u(M, \partial M)$ are locally contractible. 
\end{itemize}
\end{proposition} 

For $n$-submanifolds of $M$ we have the following conclusions. 

\begin{corollary}\label{cor_LD} \mbox{} 
\begin{enumerate}
\item Suppose $M = A \cup B$, $A$, $B$ are $n$-manifolds and $A - B \subset_u A$. Then, \\ 
\hspace*{10mm} $(A, d|_A)$, $(B, d|_B)$ $:$ $\LD$ \LRA $(M,d)$ $:$ $\LD$. 

\item Suppose $L \subset M$, $L$ is a $n$-manifold and $M - L$ is relatively compact in $M$. Then
\begin{itemize}
\item[(i)\,] $(L,d|_L) : \LD$ \LRA $(M,d) : \LD$, 
\item[(ii)] if $L$ is closed in $M$, then $(M,d) : \LD$ \LRA $(L,d|_L) : \LD$
\end{itemize}

\item Suppose $M = K \cup \bigcup_{i=1}^m L_i$, $K$ is compact, each $L_i$ is an $n$-manifold and $d(L_i, L_j) > 0$ for any $i \neq j$. 
Then, 
\begin{itemize}
\item[(i)\,] $(L_i, d|_{L_i}) : \LD$ $(i=1, \cdots, m)$ \LRA $(M,d) : \LD$, 
\item[(ii)] if each $L_i$ is closed in $M$, then $(M,d) : \LD$ \LRA $(L_i, d|_{L_i}) : \LD$ $(i=1, \cdots, m)$. 
\end{itemize}
\end{enumerate}
\end{corollary}

\begin{example}\label{example_End} (Ends of manifolds) 

(1) Suppose $(M, d)$ is a connected metric $n$-manifold and 
$N$ is a compact $n$-submanifold of $M$. We assume that ${\rm Fr}_M N$ is locally flat and transversal to $\partial M$ so that 
${\rm Fr}_M N$ is a proper $(n-1)$-submanifold of $M$ and $L = cl_M (M - N)$ is also an $n$-submanifold of $M$. 
Then, $L$ has only finitely many connected components $K_1, \cdots, K_k$, $L_1, \cdots, L_l$, 
where $K_i$'s are compact and $L_j$'s are non-compact. 
The enlargement $N_1 = N \cup \big( \cup_{i=1}^k K_i \big)$ is also a compact $n$-submanifold of $M$  
(which is connected if $N$ is connected)  and $M = N_1 \cup \big( \cup_{j=1}^l L_j \big)$. 
From Corollary~\ref{cor_LD}\,(3) it follows that, whenever $d(L_i, L_j) > 0$ for any $i \neq j$, 
$$\mbox{$M : \LD$ \LLRA $L_j : \LD$ $(j=1, \cdots, l)$.}$$ 

(2) Suppose $N$ is a compact $n$-manifold with nonempty boundary and 
$C_i$ $(i = 1, \cdots, m)$ is a non-empty collection of connected components of the boundary $\partial N$. 
For each $C_i$ take a collar $E_i = C_i \times [0,1]$ of $C_i = C_i \times \{ 1 \}$ in $N$ such that $E_i \cap E_j = \emptyset$ $(i \neq j)$. 
Consider the non-compact $n$-manifold $M = N - \cup_{i=1} C_i$. 
The ends of $M$ are in 1-1 correspondence with the product ends $L_i := E_i - C_i = C_i \times [0,1)$ $(i = 1, \cdots, m)$. 
Suppose $M$ is equipped with a metric $d$ such that $d(L_i, L_j) > 0$ for any $i \neq j$. Then 
$M : \LD$ if and only if each $L_i : \LD$. 
\end{example}

\begin{example}\label{exp_collar} (Boundary collars) 

Suppose $(M, d)$ is a metric $n$-manifold and $N \subset M$ is an $n$-manifold. 
Then $(N, d|_N)$ : (LD) if $M$ : (LD) and 
there exists a collar $E = \partial N \times [0,3]$ of $\partial N = \partial N \times \{ 0 \}$ in $N$ such that 
\begin{enumerate}
\item[(i)\,] $(E, d|_E)$ : (LD) and 
\item[(ii)] $\partial N \times [0,1] \subset_u \partial N \times [0,2]  \subset_u E$ in $N$ \ and \ 
$N - (\partial N \times [0, 1)) \subset_u N$ in $M$.  
\end{enumerate}
In fact, we can write $N = E_1 \cup F$, 
where $E_r =\partial N \times [0,r]$ $(0 < r \leq 3)$ and $F = N - (\partial N \times [0, 1))$. 
Since $M$ : (LD) and $F \subset_u N \subset M$, it follows that 
$F$ : (LD) in $M$ and in $N$. 
On the other hand, since $E$ : (LD) and $E_2 \subset_u E \subset N$, we have $E_2$ : (LD) in $E$ and in $N$. 
Since $E_1 \subset_u E_2 \subset N$, by Proposition~\ref{prop_LD}(3)(i) $N = E_1 \cup F$ : (LD) in $N$. 
\end{example}


\subsection{Proof of Proposition~\ref{prop_LD}} \mbox{}

\begin{lemma}\label{lem_LD_1} Suppose ${\cal X} = (X, W', W, Z, Y) \in {\cal S}(M, d)$. 
\begin{enumerate} 
\item Suppose $\LD_{(M, d)}({\cal X})$ holds. 
\begin{itemize} 
\item[(i)\ ] If $W' \subset W_1' \subset W \subset M$, then $\LD_{(M, d)}(X, W_1', W, Z, Y)$ holds.  
\item[(ii)\,] If $W' \subset_u W \subset W_1 \subset M$, then $\LD_{(M, d)}(X, W', W_1, Z, Y)$ holds.  
\item[(iii)] If $Z_1 \subset Z$ and $Y \subset Y_1 \subset M$, then $\LD_{(M, d)}(X, W', W, Z_1, Y_1)$ holds.  
\end{itemize} 

\item Suppose $W \subset_u N \subset M$ and $N$ is an $n$-manifold.  
Then 
\begin{itemize}
\item[(i)\,] ${\cal X}|_N := (X, W', W, Z \cap N, Y \cap N) \in {\cal S}(N, d|_N)$ and 
\item[(ii)] $\LD_{(N, d|_N)}({\cal X}|_N)$ \LLRA $\LD_{(M, d)}({\cal X})$. 
\end{itemize} 

\item Suppose $h : (M,d) \approx (M', d')$ is a uniform homeomorphism. 
Then 
\begin{itemize}
\item[(i)\,] $h({\cal X}) := (h(X), h(W'), h(W), h(Z), h(Y)) \in {\cal S}(M',d')$ and 
\item[(ii)] $\LD_M({\cal X})$ \LLRA $\LD_{M'}(h({\cal X}))$. 
\end{itemize} 

\end{enumerate} 
\end{lemma} 

\begin{proof} 
(1)  By the assumption $\LD_M({\cal X})$ we have 
a neighborhood $\W=\E^u_\ast(i_W, \gamma, W, M; Y)$ of $i_W$ in $\E^u_\ast(W, M; Y)$ and 
an admissible deformation $\phi : \W \times [0,1] \lra \E^u_\ast(W, M; Z)$ for ${\cal X}$. 
In each of the cases (i), (ii) and (iii) the required admissible deformation $(\W_1, \psi)$ is defined as follows: 

(i) Let $(\W_1, \psi) = (\W, \phi)$. 

(ii) Let $\W_1 := \E^u_\ast(i_{W_1}, \gamma; W_1, M; Y)$ and define $\psi$ by  
$$\psi : \W_1 \times [0,1] \lra \E^u_\ast(W_1, M; Z) : \ \ 
\psi_t(f) = 
\left\{ \begin{array}[c]{@{\ }ll}
\phi_t(f|_{W}) & \text{on $W$} \\[2mm] 
f & \text{on $W_1 - W'$}. 
\end{array}
\right.$$ 
\vskip 1mm 
\noindent One can easily check that $\psi$ is well-defined and satisfies the admissibility condition. 
The uniform continuity of $\psi_t(f)$ follows from those of $\phi_t(f|_{W})$ and $f$ and $W_1 - W \subset_u W_1 - W' \subset W_1$. 
For the inverse map $\psi_t(f)^{-1} : f(W_1) \to W_1$, consider the factorization  
$\psi_t(f)^{-1} = (\psi_t(f)^{-1}f)f^{-1}$. 
Since 
$$\psi_t(f)^{-1}f = 
\left\{ \begin{array}[c]{@{\ }ll}
\phi_t(f|_{W})^{-1} f|_{W} & \text{on $W$} \\[2mm] 
\id & \text{on $W_1 - W'$}, 
\end{array}
\right.$$ 
\vskip 1mm 
\noindent the uniform continuity of $\psi_t(f)^{-1}f$ and $\psi_t(f)^{-1}$ follows from those of $\phi_t(f|_{W})^{-1}$, $f$ and $f^{-1}$.   

(iii) Since $\E^u_\ast(W, M; Y_1) \subset \E^u_\ast(W, M; Y)$ and $\E^u_\ast(W, M; Z) \subset \E^u_\ast(W, M; Z_1)$, 
we can define $(\W_1, \psi)$ by $\W_1 = \W \cap \E^u_\ast(W, M; Y_1)$ and 
$$\psi : \W_1 \times [0,1] \lra  \E^u_\ast(W, M; Z_1) : \ \psi_t(f) = \phi_t(f)$$  

(2)(i) In general, if $f : (M', d') \to (M, d)$ is a uniformly continuous map, then 
$$f^{-1}({\cal X}) = (f^{-1}(X), f^{-1}(W'), f^{-1}(W), f^{-1}(Z), f^{-1}(Y)) \in {\cal S}(M', d').$$ 

(ii) Since $W \subset_u N$, we have $O_\e(W) \subset N$ for some $\e > 0$. 
Note that $O_\e(W) \cap \partial N = O_\e(W) \cap \partial M \ (= \partial O_\e(W))$ and $W\cap \partial N = W \cap \partial M$. 
The inclusion $i_N : N \subset M$ induces 
the isometry 
$${i_N}_\ast : \E^u_\ast(i_W, \e, W, N) \ \cong \ \E^u_\ast(i_W, \e, W, M) : \ \ {i_N}_\ast(f) = i_N f.$$ 
Under this isometry 
the correspondence between 
an admissible deformation $\phi$ for ${\cal X}$ and an admissible deformation $\psi$ for ${\cal X}|_N$ 
is described in the following diagram: 
$$\xymatrix@M=4pt{
\E^u_\ast(i_W, \delta, W, M; Y) \ar[r]^{\phi_t} & \E^u_\ast(i_W, \e, W, M; Z) \ar[r]^{\hspace{3mm} \subset} & \E^u_\ast(W, M; Z) \\ 
\E^u_\ast(i_W, \delta, W, N; Y \cap N) \ar[r]_{\psi_t} \ar[u]^{{i_N}_\ast}_{\cong} &\E^u_\ast(i_W, \e, W, N; Z \cap N) \ar[u]_{{i_N}_\ast}^{\cong} \ar[r]^{\hspace{3mm} \subset} & \E^u_\ast(W, N; Z \cap N)
}$$

(3)(ii) The uniform homeomorphism $h$ induces a homeomorphism 
$$h_{\#} : \E^u_\ast(W, M) \cong \E^u_\ast(h(W), M') : \ h_{\#}(f) = hfh^{-1}.$$
The next diagram represents the correspondence between 
admissible deformations $(\W, \phi)$ for ${\cal X}$ and $(\W', \psi)$ for $h({\cal X})$ :  
\vspace{-2mm} 
$$\begin{array}[c]{rc@{\ }c@{\ }cl}
& & & \phi_t & \\
\E^u_\ast(W, M; Y) & \supset & {\cal W} & \lra & \E^u_\ast(W, M; Z) \\[2mm] 
h_{\#} \big \downarrow \approx \hspace{5mm} & & h_{\#} \big \downarrow \approx \ & & \hspace{5mm} h_{\#} \big \downarrow \approx \\[2mm]  
\E^u_\ast(h(W), M'; h(Y)) & \supset & {\cal W}'  & \lra & \E^u_\ast(h(W), M'; h(Z)). \\ 
& & & \psi_t & 
\end{array}$$
This completes the proof. 
\end{proof} 

\begin{lemma}\label{lem_LD_2} Suppose $X \subset_u N \subset M$ and $N$ is an $n$-manifold.  
Then, $\LD_N(X)$ \LLRA $\LD_M(X)$. 
\end{lemma} 

\begin{proof} Since $X \subset_u N$, there exists $\e > 0$ with $O_\e(X) \subset N$. 

(1) $\LD_N(X)$ \LRA $\LD_M(X)$: 

Take any tuple ${\cal X} = (X, W', W, Z, Y) \in {\cal S}(M)$. 
Choose $\delta \in (0, \e/2)$ with $O_{2\delta}(X) \subset W'$ and 
consider the tuple ${\cal Y} = (X, O_\delta(X), O_{2\delta}(X), Z, Y) \in {\cal S}(M)$.  
We have $\LD_N({\cal Y}|_N)$ by $\LD_N(X)$ and 
$\LD_M({\cal Y})$ by Lemma~\ref{lem_LD_1}\,(2) and $O_{2\delta}(X) \subset_u O_\e(X) \subset N \subset M$. 
Hence $\LD_M({\cal X})$ follows from Lemma~\ref{lem_LD_1}\,(1). 

(2) $\LD_M(X)$ \LRA $\LD_N(X)$: 

Take any tuple ${\cal Y} = (X, W', W, Z, Y) \in {\cal S}(N)$. 
There exists $\delta \in (0, \e/2)$ such that $O_{2\delta}(X, N) \subset W'$ and $O_{\delta}(Z, N) \subset Y$.  
Note that $O_{2\delta}(X, N) = O_{2\delta}(X) \subset_u N$ and $O_\delta(Z, N) = O_\delta(Z) \cap N$. 
Consider the tuple ${\cal X} = (X, O_\delta(X), O_{2\delta}(X), Z, O_{\delta}(Z)) \in {\cal S}(M)$. 
Then,  we have $\LD_M({\cal X})$ by $\LD_M(X)$ and $\LD_N({\cal X}|_N)$ by Lemma~\ref{lem_LD_1}\,(2).   
Hence Lemma~\ref{lem_LD_1}\,(1) implies $\LD_N({\cal Y})$ as required. 
\end{proof}

\begin{proof}[\bf Proof of  Proposition~\ref{prop_LD}] \mbox{} 

(1) The assertion follows from Lemma~\ref{lem_LD_1}(3). 

(2) The statement (i) is obvious. 

(ii) For any $X \subset A$, it follows that 
$X \subset_u N \subset M$ and $\LD_N(X)$ \LLRA $\LD_M(X)$ by Lemma~\ref{lem_LD_2}.
 
(3)(i) Take any ${\cal X} = (X, W', W, Z, Y) \in {\cal S}(M, d)$ with $X \subset A\cup B$. 
We have to find a neighborhood $\W$ of $i_W$ in $\E^u_\ast(W, M; Y)$ and 
an admissible deformation $\phi : \W \times [0,1] \lra \E^u_\ast(W, M; Z)$ for ${\cal X}$. 
There exists $\e > 0$ such that $O_{2\e}(A) \subset U$, $O_{2\e}(X) \subset W'$ and $O_{2\e}(Z) \subset Y$.  

(a) Let $X_1 = O_\e(X \cap A)$ and $Z_1 = O_\e(Z)$.    
Then, ${\cal Y} = (X_1, W', W, Z_1, Y) \in {\cal S}(M, d)$ and $X_1 \subset U$. 
Since $U$ : (LD)$_M$, there exists a neighborhoodd $\W_1$ of $i_W$ in $\E^u_\ast(W, M; Y)$ and 
an admissible deformation for ${\cal Y}$, 
$$\psi : \W_1 \times [0,1] \lra \E^u_\ast(W, M; Z_1).$$   
 
(b) Let $X_2 = X \cap B$, $Z_2 = Z\cup (X \cap A)$ and $Y_2 = Z_1 \cup X_1$. 
Then $O_\e(Z_2) = O_\e(Z) \cup O_\e(X \cap A) = Z_1 \cup X_1 = Y_2$, so that 
${\cal Z} = (X_2, W', W, Z_2, Y_2) \in {\cal S}(M, d)$ and $X_2 \subset B$. 
Since $B$ : (LD)$_M$, there exists a neighborhood $\W_2$ of $i_W$ in $\E^u_\ast(W, M; Y_2)$ and 
an admissible deformation for ${\cal Z}$, 
$$\chi : \W_2 \times [0,1] \lra \E^u_\ast(W, M; Z_2).$$  

Since $\psi_1(f) = \id$ on $(W \cap Z_1) \cup X_1 = W \cap Y_2$ for any $f \in \W_1$, we have 
the map 
$$\psi_1 : \W_1 \to \E^u_\ast(W, M; Y_2).$$ 
Since $\psi_1(i_W) = i_W \in \W_2$, 
we can find a neighborhood $\W$ of $i_W$ in $\W_1$ such that $\psi_1(\W) \subset \W_2$. 
Finally, the required admissible deformation for ${\cal X}$ is defined by 
$$\phi : \W \times [0,1] \lra \E^u_\ast(W, M; Z) : \hspace{2mm} 
\phi_t(f) = 
 \left\{ 
\begin{array}[c]{@{\ }ll}
\psi_{2t}(f) & (t \in [0,1/2]) \\[2mm]
\chi_{2t-1}(\psi_1(f)) & (t \in [1/2,1]).  
\end{array} \right.$$ 
Note that $\phi_1(f) = \chi_1(\psi_1(f)) = \id$ on $(W \cap Z_2) \cup X_2 \supset(X \cap A) \cup (X \cap B) =  X$. 

The statement (ii) follows from (i).

(4) $(\Longleftarrow)$: The assertion follows from (2)(i). 

$(\Longrightarrow)$: Take any compact neighborhood $L$ of $cl_MK$ in $M$. Then 
$K \subset_u L$ and $L$ : $\LD_M$ by Example~\ref{exp_EK_covering} (1). 
Thus, $A \cup K$ : $\LD_M$ by (3)(i).  

(5) Suppose the tuple $(M, M, M, \emptyset, \emptyset) \in {\cal S}(M, d)$ has  an admissible deformation $(\W, \phi)$. 
Then $\W$ is a neighborhood of $\id_M$ in $\E^u_\ast(M, M)$ and 
$\phi_t : \W \to \E^u_\ast(M, M)$ satisfies the following conditions: 
\begin{enumerate}
\item[(i)\,] For any $f \in \W$ 
\begin{tabular}[t]{l}
(a) $\phi_0(f) = f$, \ \ (b) $\phi_1(f) = \id_M$, \ \ (c) $\phi_t(f)(M) = f(M)$ \ \ and \\[2mm] 
(d) if $f = \id$ on $\partial M$, then $\phi_t(f) = \id$ on $\partial M$. 
\end{tabular}
\item[(ii)] $\phi_t(\id_M) = \id_M$. 
\end{enumerate}

Since ${\cal H}^u(M) \subset \E^u_\ast(M, M)$, 
it follows that $\U := \W \cap {\cal H}^u(M)$ is a neighborhood of $\id_M$ in ${\cal H}^u(M)$ 
and $\phi_t(\U) \subset {\cal H}^u(M)$ by (i)(c). 
Then, the restriction $\phi_t : \U \to {\cal H}^u(M)$ is a contraction of $\U$ in ${\cal H}^u(M)$ rel $\id_M$. 
By (i)(d) this contraction also restricts to a local contraction in ${\cal H}^u(M, \partial M)$. 
\end{proof} 

\begin{proof}[\bf Proof of  Corollary~\ref{cor_LD}] \mbox{} 

(1) Let $A_1 = A - B$ and take $\e > 0$ with $O_{3\e}(A_1) \subset A$. 
Since $O_{2\e}(A_1) \subset_u A \subset M$ and $A$ : $\LD_A$, 
it follows that $O_{2\e}(A_1)$ : $\LD_A$ and $\LD_M$. 
Let $B_1 = B - O_{\e}(A_1)$. Since $O_{\e}(B_1) \subset B \subset M$ and \break $B$ : $\LD_B$, 
we have that $B_1$ : $\LD_B$ and $\LD_M$. 
Therefore, by Proposition~\ref{prop_LD} (3)(i) it follows that $M = O_{\e}(A_1) \cup B_1$ : $\LD_M$. 

(2) Since $K = cl_M(M - L)$ is compact, there exists a compact neighborhood $C$ of $K$ in $M$ and   
$\e > 0$ such that $O_\e(K) \subset C$. 
Let $L_1 = L - C$. 
Since $O_\e(L_1) \subset L$, $C$ is compact and $M = L_1 \cup C$, from Proposition~\ref{prop_LD} (2), (4)
it follows that 
$$\mbox{$L$ : $\LD_L$ \LRA $L_1$ : $\LD_L$ \LLRA $L_1$ : $\LD_M$ \LLRA $M$ :  $\LD_M$.}$$ 
When $L$ is closed in $M$, since $C \cap L$ is compact and $L = L_1 \cup (C \cap L)$, it follows that  
$$\mbox{$L$ : $\LD_L$ \LLRA $L_1$ : $\LD_L$.}$$ 
The conclusions follows from these implications. 

(3) Let $L = \bigcup_{i=1}^m L_i$ and take $\e > 0$ such that $d(L_i, L_j) > \e$ for any $i \neq j$. 
Then, $L$ is an $n$-manifold and $O_\e(L_i; L) = L_i$ so $L_i \subset_u L_i \subset L$ $(i = 1, \cdots, m)$. 
Hence, by Proposition~\ref{prop_LD} (2), (3) 
$$\mbox{$L_i : \LD$ $(i=1, \cdots, m)$ \LLRA 
$L_i$ : $\LD_L$ $(i=1, \cdots, m)$ \LLRA $L : \LD$.}$$  
Since $M = K \cup L$ and $K$ is compact, from (2) it follows that 

{\rm (a)} $L : \LD$ \LRA $M : \LD$ \ and \ 
{\rm (b)} if $L$ is closed in $M$, then $M : \LD$ \LRA $L : \LD$. \\
The assertions follow from these observations. 
\end{proof}


\section{Examples} 

In this section we discuss some examples of metric manifolds with the property $\LD$. 

\subsection{Manifolds with geometric group actions} \mbox{} 

In this subsection we show that a metric $n$-manifold has the property $\LD$ 
if it admits a locally geometric group action (Theorem~\ref{thm_LD_geom_action}). We refer to \cite[Chapter I.8]{BH} for basic facts on geometric group actions. 
First we recall some related notions. 

Throughout this subsection, $X = (X, d)$ is a locally compact metric space, 
$G$ is a (discrete) group and 
$\Phi : G \times X \to X$ is a continuous action of $G$ on $X$. 
As usual, for $g \in G$ and $x \in X$ the element $\Phi(g,x) \in X$ is denoted by $gx$. 
For a point $x \in X$ the orbit $Gx$ and the isotropy subgroup $G_x$ of $x$ 
are defined by $Gx = \{ gx \mid g \in G \}$ and $G_x = \{ g \in G \mid gx = x \}$ respectively. 
More generally, 
for any subsets $H \subset G$ and $C \subset X$ let $HC = \{ gx \mid g \in H, x \in C \}$ and  
$G_C = \{ g \in G \mid gC = C \}$. 
Then, $G_C$ is a subgroup of $G$ and for any (left) coset $\overline{g} \in G/G_C$ the subset $gC \subset X$ is well-defined. 

The action $\Phi$ of $G$ on $X$ is called geometric if it is proper, cocompact and isometric. 
Here, $\Phi$ is (a) proper if $\{ g \in G \mid gF \cap F \neq \emptyset \}$ is a finite set for any compact subset $F$ of $X$, 
(b) cocompact if the quotient space $X/G$ is compact, and 
(c) isometric if each $g \in G$ acts on $X$ as an isometry. 
Note that (i) the action $\Phi$ is cocompact if and only if $X = GK$ for some compact subset $K$ of $X$ and 
(ii) if $\Phi$ is proper, then ($\alpha$) if $X$ is separable, then $G$ is a countable group since $X$ is $\sigma$-compact, 
($\beta$) for any nonempty compact subset $C$ of $X$ the family $\{ gC \mid g \in G \}$ is locally finite and 
$G_C$ is a finite subgroup of $G$, so ($\gamma$) the orbit $Gx$ is discrete in $X$ for any point $x \in X$. 

In this article we work in a slightly more general setting.  

\begin{defn} We say that the action $\Phi$ of $G$ on $X$ is 
\begin{enumerate}
\item locally isometric if 
for every $x \in X$ there exists $\e \in (0, \infty]$ such that 
\begin{itemize}
\item[] $(\natural)_x$ \ each $g \in G$ maps $O_\e(x)$ isometrically onto $O_\e(gx)$, and 
\end{itemize}
\item locally geometric if it is proper, cocompact and locally isometric. 
\end{enumerate}
\end{defn}

\begin{remark}\label{rem_geom_action} For $x \in X$, 
let $r_x = \sup\{ \e \in [0, \infty] \mid \e : (\natural)_x\} \in [0, \infty]$. 
Then 
\begin{itemize}
\item[(i)\ ] $r_x$ itself satisfies the condtion $(\natural)_x$, so that $\{ \e \in [0, \infty] \mid \e : (\natural)_x\} = [0, r_x]$ and 
\item[(ii)\,] the action $\Phi$ is locally isometric if and only if $r_x > 0$ for each $x \in X$, 
\item[(iii)] if $\gamma \in [0, r_x]$, then $O_\gamma(gx) = gO_\gamma(x)$ for each $g \in G$ and  
$O_\gamma(Fx) = FO_\gamma(x)$ for any subset $F \subset G$. 
\end{itemize}
\end{remark}

\begin{lemma}\label{lem_geom_action} Suppose the action $\Phi$ is locally geometric and  let $x \in X$ be any point. 
\begin{enumerate}
\item The orbit $Gx$ is uniformly discrete. Hence, there exists $\e \in (0, r_x)$ such that the orbit $Gx$ is $4\e$-discrete. 
\end{enumerate}

\begin{enumerate}
\item[(2)] Let $\Lambda$ be a complete set of representatives of cosets in $G/G_x$. 
\begin{itemize}
\item[(i)\,] $O_\e(Gx)$ is the disjoint union of open subsets $O_\e(gx) = gO_\e(x)$ $(g \in \Lambda)$ and 
the family $\{ O_\e(gx) \mid g \in \Lambda \}$ is $2\e$-discrete. 
\item[(ii)] The map $\pi : O_\e(Gx) \to O_\e(x)$ defined by 
$$\pi|_{O_\e(gx)} = g^{-1} : \ O_\e(gx) \cong O_\e(x) \hspace{5mm} \mbox{for each $g \in \Lambda$}$$ 
is a trivial metric covering projection. 
\end{itemize}
\item[(3)] In addition, assume that $X$ is a metric $n$-manifold. 
Let $D$ be a closed $n$-disk neighborhood of $x$ in $O_\e(x)$ and $\delta \in (0, \e/2)$ be such that 
$O_{2\delta}(x) \subset D$. Then $O_\delta(Gx) = GO_\delta(x)$ and it satisfies the condition $\LD_X$. 
\end{enumerate}
\end{lemma}

\begin{proof} 
(1) Since $\Phi$ is proper and locally isometric, it follows that $r_x > 0$ and the orbit $Gx$ is discrete. 
Hence, there exists $\delta \in (0, r_x)$ such that $O_\delta(x) \cap Gx = \{ x \}$. 
Then $Gx$ is $\delta$-discrete. 
In fact, if $g, h \in G$ and $gx \neq hx$, then $g^{-1}hx \neq x$ so that 
$g^{-1}hx \not\in O_\delta(x)$. 
Since $\delta \in (0, r_x)$, it follows that $gO_\delta(x) = O_\delta(gx)$ and so $hx \not\in O_\delta(gx)$. 

(2) (i) For any distinct $g, h \in \Lambda$ it follows that $gx \neq hx$, so $d(gx, hx) \geq 4\e$ since $Gx$ is $4\e$-discrete,   
hence $d(O_\e(gx), O_\e(hx)) \geq 2\e$. 

(ii) Since $\e \in (0, r_x)$, each $g \in \Lambda$ induces an isometry $g : O_\e(x) \cong O_\e(gx)$. 
Hence, the restriction $\pi|_{O_\e(gx)} = g^{-1} : O_\e(gx) \cong O_\e(x)$ is a well-defined isometry. 
For each $y \in O_\e(x)$ the fiber of $y$ is given by $\pi^{-1}(y) = \Lambda y$.  
Since $\{ O_\e(gx) \mid g \in \Lambda \}$ is $2\e$-discrete and $gy \in O_\e(gx)$ for each $g \in \Lambda$, 
it follows that $\pi^{-1}(y)$ is also $2\e$-discrete. 

For any $y, y' \in O_\e(Gx)$ we have $d(\pi(y), \pi(y')) \leq d(y, y')$. In fact, 
if $y, y' \in O_\e(gx)$ for some $g \in \Lambda$, then $d(\pi(y), \pi(y')) = d(y, y')$ since $\pi|_{O_\e(gx)}$ is an isometry, 
and if $y \in O_\e(gx)$ and $y' \in O_\e(hx)$ for some distinct $g, h \in \Lambda$, then 
$d(\pi(y), \pi(y')) \leq {\rm diam}\,O_\e(x) \leq 2\e \leq d(O_\e(gx), O_\e(hx)) \leq d(y,y')$. 

(3) Since $\delta \in (0, r_x)$, we have $O_\delta(Gx) = GO_\delta(x)$. 
Let $N = \pi^{-1}(D)$. 
Since the restriction $\pi : N \to D$ is also a metric covering projection, 
it follows that $N$ is an $n$-manifold and 
satisfies $\LD$ by Example~\ref{exp_EK_covering} (2).   
Since $O_{2\delta}(x) \subset D$ and $2\delta \in (0, r_x)$, we have 
$$O_{2\delta}(Gx) = O_{2\delta}(\Lambda x) 
= \Lambda O_{2\delta}(x) \subset \Lambda D = N.$$ 
Then, by Proposition~\ref{prop_LD}\,(2) $O_\delta(Gx)$ satisfies the condition $\LD_N$ and hence $\LD_X$. 
\end{proof}

\begin{theorem}\label{thm_LD_geom_action} 
A metric manifold has the property \LD\,if it admits a locally geometric group action. 
\end{theorem} 

\begin{proof}
Suppose a metric manifold $(M,d)$ admits a locally geometric group action $\Phi : G \times M \to M$. 
By Lemma~\ref{lem_geom_action}\,(3) and Remark~\ref{rem_geom_action}\,(iii) 
for each point $x \in M$ there exists $\e(x) > 0$ such that 
$O_{\e(x)}(Gx) = GO_{\e(x)}(x)$ and $O_{2\e(x)}(Gx)$ : $\LD_M$. 
Since the group action is cocompact, $M = GK$ for some compact subset $K \subset M$ and 
there exist finitely many points $x_1, \cdots, x_m \in K$ such that 
$K \subset \cup_{i=1}^m O_{\e(x_i)}(x_i)$.  
Then $M = GK = \cup_{i=1}^m O_{\e(x_i)}(Gx_i)$ and, 
since $O_{2\e(x_i)}(Gx_i)$ : $\LD_M$ for $i = 1, \cdots, m$, 
Proposition~\ref{prop_LD}\,(3)\,(ii) implies that $M$ : $\LD$. 
This completes the proof. 
\end{proof} 


\subsection{Typical metric ends} 

\begin{example}\label{example_Euclidean} (Euclidean ends) 

(1) The Euclidean space $\IR^n$ with the standard Euclidean metric admits 
the canonical geometric action of $\IZ^n$ 
(and the associated Riemannian covering projection $\pi : \IR^n \to \IR^n/\IZ^n$ onto the flat torus). 
Therefore, $\IR^n$ has the property $\LD$. 

(2) The half space $\IR^n_{\geq 0} = \{ \bs{x} \in \IR^n \mid x_n \geq 0 \}$ also has the property $\LD$. 
This follows from Example~\ref{exp_collar}. In fact, 
$\IR^n$ : (LD) and 
the boundary collar $E = \IR^{n-1} \times [0,3]$ of $\IR^n_{\geq 0}$ also has the property $\LD$ since it admits 
the geometric action of $\IZ^{n-1}$ defined by 
$$\mbox{${m} \cdot ({y}, z) = (y+ m, z)$ \ \ $({m} \in \IZ^{n-1}, ({y}, z) \in E)$.}$$ 

(3) By Corollary~\ref{cor_LD}\,(2) the Euclidean ends $\IR^n - O_r(0)$ $(r > 0)$ and 
the half Euclidean ends $\IR^n_{\geq 0} - O_r(0)$ $(r > 0)$ have the property $\LD$. 
\end{example}

\begin{example}\label{example_Hyperbolic} (Hyperbolic ends) 

(1) The Hyperbolic space $\IH^n$ admits a Riemannian covering projection onto a closed hyperbolic manifold. 
Hence, $\IH^n$ has the property $\LD$. 
For any $\kappa< 0$, the space form $M^n_\kappa$ has the property $\LD$, 
since it is homothetic to $\IH^n$. 
By Corollary~\ref{cor_LD}\,(2) the hyperbolic ends 
$M^n_\kappa - O_r(o)$ $(o \in M^n_\kappa, r > 0)$ also have the property $\LD$. 

(2) Consider any horosphere $\Sigma$ of $\IH^n$. Let $D$ and $E$ denote the interior and exterior of $\Sigma$. 
Then $\overline{D} = \Sigma \cup D$ and $\overline{E} = \Sigma \cup E$ have the property $\LD$. 
Hence, the ends $\overline{D} - O_r(o)$ and $\overline{E} - O_r(o)$ 
 $(o \in \IH^n, r > 0)$ also have the property $\LD$, 
 whenever they are $n$-manifolds. 
 (The exceptional case appears when $O_r(o)$ touches $\Sigma$ in $D$ and $E$ respectively.)

To see this, consider the upper half space model of $\IH^n$ (cf. \cite[\S4.6.]{Ra}), 
in which the underlying space is $\IR^n_{>0} := \{ x \in \IR^n \mid x_n > 0  \} = \IR^{n-1} \times (0, \infty)$ and the metric $d$ is given by 
$$\cosh d(x, y) = 1 + \frac{|x - y|^2}{2x_ny_n} \hspace{5mm} (x, y \in \IR^n_{>0}).$$
Then, it suffices to examine the horosphere $\Sigma_1 := \{ x \in \IR^n \mid x_n = 1 \}$,  
since any horosphere $\Sigma$ of $\IH^n$ admits an isometry $h$ of $\IH^n$ with $h(\Sigma) =  \Sigma_1$. 

The interior and exterior of $\Sigma_1$ are given by 
$D_1 = \IR^{n-1} \times (1, \infty)$ and $E_1 = \IR^{n-1} \times (0, 1)$. 
The claim $\overline{D_1}$ : (LD) follows from Example~\ref{exp_collar}. In fact, 
$\IH^n$ : (LD) and 
the collar $\IR^{n-1} \times [1,4]$ of $\Sigma_1 = \IR^{n-1} \times \{ 1 \}$ in $\overline{D_1} = \IR^{n-1} \times [1,\infty)$ has the property $\LD$ since it admits 
the geometric action of $\IZ^{n-1}$ defined by 
$$\mbox{${m} \cdot ({y}, z) = (y+ m, z)$ \ \ $({m} \in \IZ^{n-1}, ({y}, z) \in \IR^{n-1} \times [1,4])$.}$$ 
The condition (ii) in Example~\ref{exp_collar} follows from the fact that 
$$\cosh d(\IR^{n-1} \times \{ r \}, \IR^{n-1} \times \{ s \}) = 1 + \frac{(r-s)^2}{2rs} \hspace{5mm} (r, s > 0).$$ 

A similar argument shows that $\overline{E_1}$ : $\LD$. 
\end{example}

\begin{example}\label{example_Cylinder} (Cylindrical ends) \\
Suppose $(N, \rho)$ is a compact metric $n$-manifold. 

(1) The cylinder over $(N, \rho)$ is the metric space $(M ,d)$, 
where $M = N \times \IR$ and the metric $d$ is defined by 
$$d((x,s), (y,t)) = \sqrt{\rho(x,y)^2 + |s-t|^2} \ \ \ ((x,s), (y,t) \in M).$$ 
The group $\IZ$ acts on $(M ,d)$ geometrically by 
$$m \cdot (x,s) = (x, s+m) \ \ \ ((x,s) \in M, m \in \IZ).$$
Therefore, the cylinder $(M ,d)$ has the property $\LD$. 

(2) The cylindrical end over $(N, \rho)$ is the half product $N \times [0, \infty)$ with the metric $d$ defined in (1).
It has the property $\LD$ by (1) and Corollary~\ref{cor_LD}\,(3). 
\end{example}

In Example~\ref{example_End}\,(2), if each $(L_i, d)$ is an end of  either Euclidean, hyperbolic or cylindrical type, 
then $M$ has the property $\LD$. 


\subsection{The $\kappa$-cone ends $(\kappa \leq 0)$ over compact Lipschitz metric manifolds} \mbox{} 

In this subsection we extend the results in Examples~\ref{example_Euclidean} and ~\ref{example_Hyperbolic}
to the case of $\kappa$-cone ends $(\kappa \leq 0)$ over compact Lipschitz metric manifolds.  

By Example~\ref{exp_cone_sph} the $\kappa$-cone $C_\kappa({\Bbb S}^n)$ is isometric to $M_\kappa^{n+1}$ for $\kappa \leq 0$ and $C_0({\Bbb S}^n_{\geq 0})$ is isometric to $\IR^{n+1}_{\geq 0}$. 
Examples~\ref{example_Euclidean}, ~\ref{example_Hyperbolic} imply the following basic conclusion. 

\begin{lemma}\label{lem_cone_sph} 
\begin{itemize}
\item[] \hspace*{-14mm} {\rm (1)} The $0$-cones $C_0({\Bbb S}^n)$, $C_0({\Bbb S}^n_{\geq 0})$ and their ends $C_0({\Bbb S}^n)_r$,$C_0({\Bbb S}^n_{\geq 0})_r$ $(r > 0)$ have the property \LD. 
\item[(2)] For any $\kappa < 0$, the $\kappa$-cone $C_\kappa({\Bbb S}^n)$ and its ends $C_\kappa({\Bbb S}^n)_r$ $(r > 0)$ have the property \LD. 
\end{itemize}
\end{lemma} 

Below we extend this ${\Bbb S}^n$-case to the case of any compact Lipschitz metric manifold.

\begin{lemma}\label{lem_Eucliedan} 
\begin{itemize}
\item[] \hspace*{-14mm} {\rm (1)} Suppose $K \subset V \subset \IR^n_{\geq 0}$, $K$ is compact and $V$ is open. Then 
$C_0(K)_1 : \LD$ in $C_0(V)^\times$. 

\item[(2)] Suppose $K \subset V \subset \IR^n$, $K$ is compact and $V$ is open. Then, 
$C_\kappa(K)_1 : \LD$ in $C_\kappa(V)^\times$ for any $\kappa < 0$.  
\end{itemize}
\end{lemma}

\begin{proof} 
(1) Choose $r > 0$ such that $K \subset_u B \equiv \{ x \in \IR^n_{\geq 0} \mid \| x \| \leq r \} \subset \IR^n_{\geq 0}$ and 
let $\partial_- B := \{ x \in B \mid x_n = 0 \}$. 
Consider the closed $n$-disk with corner, $C \equiv \{ y \in {\Bbb S}^n_{\geq 0} \mid y_1 \geq 0 \}$ in ${\Bbb S}^n_{\geq 0}$ 
and let $\partial_- C \equiv \{ y \in C \mid y_{n+1} = 0 \}$. 
There exists a diffeomorphism $h : (B, \partial_- B) \approx (C, \partial_- C)$. 
Let $L = h(K)$. 
Since $h$ is a Lipschitz homeomorphism, by Lemma~\ref{lem_Lip} we have   
a Lipschitz homeomorphism of pairs, 
$$C(h) : (C_0(B)^\times, C_0(K)_1) \approx (C_0(C)^\times, C_0(L)_1).$$ 

By Lemma~\ref{lem_cone_sph} (1) $C_0({\Bbb S}^n_{\geq 0})$ : \LD. 
Since $L \subset_u C \subset {\Bbb S}^n_{\geq 0}$ and 
$C_0(L)_1 \subset_u C_0(C)^\times \subset C_0({\Bbb S}^n_{\geq 0})$, we have 
$C_0(L)_1$ : (LD) in $C_0({\Bbb S}^n_{\geq 0})$ and  in $C_0(C)^\times$.
Hence, by Proposition~\ref{prop_LD} (1) $C_0(K)_1$ : (LD) in $C_0(B)^\times$. 
Then,  $C_0(K)_1$ : (LD) in $C_0(\IR^n_{\geq 0})^\times$ and in $C_0(V)^\times$, since 
$C_0(K)_1 \subset_u C_0(B)^\times \subset C_0(\IR^n_{\geq 0})^\times$ 
and $C_0(K)_1 \subset_u C_0(V)^\times \subset C_0(\IR^n_{\geq 0})^\times$. 

(2) The argument is same as (1) with using Lemma~\ref{lem_cone_sph} (2).  
Choose $r > 0$ such that $K \subset_u B \equiv \{ x \in \IR^n \mid \| x \| \leq r \} \subset \IR^n$. 
Let $C := \IS^n_{\geq 0} \subset \IS^n$. 
Take a diffeomorphism $h : B \approx C$ and let $L = h(K)$. 
Since $h$ is a Lipschitz homeomorphism, by Lemma~\ref{lem_Lip} we have   
a Lipschitz homeomorphism of pairs, 
$$C(h) : (C_\kappa(B)^\times, C_\kappa(K)_1) \approx (C_\kappa(C)^\times, C_\kappa(L)_1).$$ 
By Lemma~\ref{lem_cone_sph} (2) $C_\kappa({\Bbb S}^n)$ : \LD. 
Since $L \subset_u C \subset {\Bbb S}^n$ and 
$C_\kappa(L)_1 \subset_u C_\kappa(C)^\times \subset C_\kappa({\Bbb S}^n)$, we have 
$C_\kappa(L)_1$ : (LD) in $C_\kappa({\Bbb S}^n)$ and  in $C_\kappa(C)^\times$.
Hence, by Proposition~\ref{prop_LD} (1) $C_\kappa(K)_1$ : (LD) in $C_\kappa(B)^\times$. 
Then,  $C_\kappa(K)_1$ : (LD) in $C_\kappa(\IR^n)^\times$ and in $C_\kappa(V)^\times$, since 
$C_\kappa(K)_1 \subset_u C_\kappa(B)^\times \subset C_\kappa(\IR^n)^\times$ 
and $C_\kappa(K)_1 \subset_u C_\kappa(V)^\times \subset C_\kappa(\IR^n)^\times$. 
\end{proof} 

\begin{proposition}\label{prop_k-cone_LD} Suppose $(N,d)$ is a compact Lipschitz metric manifold. 
Then, the $\kappa$-cone end $C_\kappa(N,d)_1$ has the property $\LD$ when 
{\rm (i)} $\kappa = 0$ or {\rm (ii)} $\kappa < 0$ and $\partial N = \emptyset$. 
\end{proposition}

\begin{proof} 
For each $x \in N$ there exists a Lipschitz homeomorphism $h_x : U_x \approx V_x$ between 
an open neighborhood $U_x$ of $x$ in $N$ and an open subset $V_x$ of $\IR^n_{\geq 0}$. 
(Note that $V_x \subset \IR^n_{>0}$ in the case (ii).) 
Take a compact neighborhood $D_x$ of $x$ in $V_x$ and let $E_x := h_x(D_x)$. 
By Lemma~\ref{lem_Lip} 
$h_x$ induces a Lipschitz homeomorphism of pairs 
$$C(h_x) : (C_\kappa(U_x)^\times, C_\kappa(D_x)_1) \approx (C_\kappa(V_x)^\times, C_\kappa(E_x)_1).$$ 

By Lemma~\ref{lem_Eucliedan} $C_\kappa(E_x)_1$ : (LD) in $C_\kappa(V_x)^\times$ 
and by Proposition~\ref{prop_LD} (1) we have 
$C_\kappa(D_x)_1$ : (LD) in $C_\kappa(U_x)^\times$. 
Since $C_\kappa(D_x)_1 \subset_u C_\kappa(U_x)^\times \subset C_\kappa(N)^\times$, 
it follows that $C_\kappa(D_x)_1$ : (LD) in $C_\kappa(N)^\times$. 

For each $x \in N$ take a compact neighborhood $C_x$ of $x$ in ${\rm Int}_N D_x$. 
Since $C_x \subset_u D_x$ in $N$, 
we have $C_\kappa(C_x)_2 \subset_u C_\kappa(D_x)_1$ in $C_\kappa(N)^\times$. 
Since $N$ is compact, there exist $x_1, \cdots, x_m \in N$ such that $N = \cup_{i=1}^m C_{x_i}$.
Then $C_\kappa(N)_2 = \cup_{i=1}^m C_\kappa(C_{x_i})_2$ and 
by Proposition~\ref{prop_LD} (3)(ii) $C_\kappa(N)_2$ : (LD) in $C_\kappa(N)^\times$. 
Since $C_\kappa(N)_2 \subset_u C_\kappa(N)_1 \subset C_\kappa(N)^\times$, it follows that 
$C_\kappa(N)_2$ : (LD) in $C_\kappa(N)_1$ 
and that $C_\kappa(N)_1$ : (LD) by Proposition~\ref{prop_LD} (4). 
\end{proof}

\section{End deformation property for uniform embeddings} 

In this section we introduce the notion of end deformation property for uniform embeddings (ED) in proper product ends and study its basic nature. 


\subsection{Definition and basic properties} \mbox{} 

\begin{defn}\label{def_proper product end} 
An $n$-dimensional proper product end is a metric $n$-manifold $(L,d)$ such that (i) the metric $d$ is proper and 
(ii) there exists a homeomorphism $\theta : S \times [1, \infty) \approx L$ for some compact $(n-1)$-manifold $S$. 
\end{defn}

Suppose $(L,d)$ is a proper product end. 
A subset $F$ of $L$ is said to be cofinal if $L - F$ is relatively compact in $L$. 
By ${\cal C \cal F}(L)$ we denote the collection of cofinal closed subsets of $L$. 
Note that ${\rm Fr}_L F$ is compact for any $F \in {\cal C \cal F}(L)$. 
We fix a homeomorphism $\theta : S \times [1, \infty) \approx L$ and 
let $L_I := \theta(S \times I)$ for $I \subset [1, \infty)$ and $L_r := L_{[r, \infty)}$ for $r \geq 1$. 
Note that (i) $F \subset L$ is cofinal if and only if $L_r \subset F$ for some $r \in [1, \infty)$ and 
(ii) $d(L_{\{ 1 \}}, L_r) \to \infty$ \ $(r \to \infty)$ since the metric $d$ is proper.  

\begin{defn}\label{def_(ED)} Suppose $(L, d)$ is a proper product end. 
We say that $(L, d)$ has the property (ED) and write $(L, d)$ :  (ED) if the following condition is satisfied: 
\begin{itemize} 
\item[$(\ast)$] For any $F \in {\cal C \cal F}(L)$ and any $\alpha > 0$ there exist $\beta > 0$, $H \in {\cal C \cal F}(L)$ with $H \subset F$ 
and an admissible deformation over $H$ (in the sense of Definition~\ref{def_adm_def})
$$\phi : \E^u_\ast(i_F, \alpha; F, (L, d)) \times [0,1] \lra \E^u_\ast(i_F, \beta; F, (L, d))$$
such that 
\begin{itemize}
\item[$(\ast)_1$] $\phi_t(f) = f$ \ on \ ${\rm Fr}_L F$ \ \ for each $(f,t) \in \E^u_\ast(i_F, \alpha; F, (L, d)) \times [0,1]$.
\end{itemize} 
\end{itemize}
\end{defn}

\begin{remark} 
In Definition~\ref{def_(ED)} the admissible deformation $\phi$ satisfies the following condition : 
\begin{itemize}
\item[] 
\begin{itemize}
\item[$(\ast)_2$] \ $\phi_t(f)(F) = f(F)$ \ \ for each $(f,t) \in \E^u_\ast(i_F, \alpha; F, (L, d)) \times [0,1]$.
\end{itemize} 
\end{itemize} 
\end{remark}

\begin{proof} 
Since ${\rm Int}_L F$ is an open subset of $L$, it follows that ${\rm Int}_L F$ is an $n$-manifold with 
$\partial \,{\rm Int}_L F = ({\rm Int}_L F) \cap \partial L$.
Let $U_\lambda$ $(\lambda \in \Lambda)$ denote the collection of connected components of ${\rm Int}_L F$, 
which forms a disjoint open covering of ${\rm Int}_L F$. 
Since $d(f, i_F) < \infty$ and $f$ is a proper embedding, it follows that 
$f(F)$ is a closed subset of $L$, 
${\rm Fr}_L f(F) = f({\rm Fr}_L F)$, ${\rm Int}_L f(F) = f({\rm Int}_L F)$ and 
$f(U_\lambda)$ $(\lambda \in \Lambda)$ is the collection of connected components of ${\rm Int}_L f(F)$, 
which forms a disjoint open covering of ${\rm Int}_L f(F)$. 
Since $L - f({\rm Fr}_L F)$ is the disjoint union of open subsets $L - f(F)$ and $f(U_\lambda)$ $(\lambda \in \Lambda)$, 
each $f(U_\lambda)$ is an open and closed connected subset of  $L - f({\rm Fr}_L F)$. 

We can apply the same argument to $\phi_t(f)$ $(t \in [0,1])$. 
From the condition $(\ast)_1$ it follows that 
$\phi_t(f)(U_\lambda)$ $(\lambda \in \Lambda)$ is a disjoint family of open and closed connected subsets of $L - f({\rm Fr}_L F)$. 
For each $\lambda \in \Lambda$ take any point $x_\lambda \in U_\lambda$. Then 
$\phi_t(f)(x_\lambda)$ $(t \in [0,1])$ is a path in $L - f({\rm Fr}_L F)$ which meets $f(U_\lambda)$. 
Hence, for each $t \in [0, 1]$, we have $\phi_t(f)(x_\lambda) \in f(U_\lambda)$ and 
$\phi_t(f)(U_\lambda) = f(U_\lambda)$ $(t \in [0,1])$.  
This implies $\phi_t(f)(F) = f(F)$ $(t \in [0,1])$. 
\end{proof} 

\begin{lemma} 
Suppose $(L,d)$, $(L',d')$ are proper product ends and $h : (L,d) \approx (L',d')$ is a uniform, coarsely uniform homeomorphism. 
If $(L,d) : \ED$, then so is $(L',d')$. 
\end{lemma}

\begin{proof} 
Given any $F' \in {\cal C \cal F}(L')$ and $\alpha' > 0$. 
Since $h^{-1}$ is coarsely uniform, there exist $\alpha > \alpha_0 > 0$ such that 
if $x',y' \in L'$ and $d'(x',y') < \alpha'$, then $d(h^{-1}(x'), h^{-1}(y')) < \alpha_0$. 
Since $(L,d) : \ED$, for $F := h^{-1}(F') \in {\cal C \cal F}(L)$ and $\alpha$ 
there exist $\beta > 0$, $H \in {\cal C \cal F}(L)$ with $H \subset F$ 
and an admissible deformation 
$$\phi : \E^u_\ast(i_F, \alpha; F, (L, d)) \times [0,1] \lra \E^u_\ast(i_F, \beta; F, (L, d))$$
over $H$ such that \ $\phi_t(f) = f$ \ on ${\rm Fr}_L F$ \ for each $(f,t) \in \E^u_\ast(i_F, \alpha; F, L) \times [0,1]$. 

Since $h$ is coarsely uniform, there exist $\beta' > \beta'_0 > 0$ such that 
if $x,y \in L$ and $d(x,y) < \beta$, then $d'(h(x),h(y)) < \beta'_0$. 
Let $H' := h(H) \in {\cal C \cal F}(L')$. 

Consider the homeomorphism 
$$\eta : \E^u_\ast(F', L') \approx \E^u_\ast(F, L) : \hspace{2mm} \eta(f) = h^{-1}f (h|_F).$$ 

Since \ \ 
$\eta(\E^u_\ast(i_{F'}, \alpha'; F', L')) \subset  \E^u_\ast(i_F, \alpha; F, L)$ \ and \ 
$\eta^{-1}(\E^u_\ast(i_F, \beta; F, L)) \subset \E^u_\ast(i_{F'}, \beta'; F', L')$, \\ 
the required admissible deformation over $H'$ is defined by 
$$\psi_t = \eta^{-1} \phi_t \eta : \E^u_\ast(i_{F'}, \alpha'; F', (L', d')) \lra \E^u_\ast(i_{F'}, \beta'; F', (L', d')) \hspace{2mm} (t \in [0,1]).$$ 
\vskip -8mm 
\end{proof} 


\begin{example}\label{exp_ED=>ED}
Suppose $(N, d)$ is a compact metric manifold. If $C_0(N, d)_1 : \LD$, then $C_0(N, d)_1 : \ED$. 
In particular, the 0-cone end $C_0(N, d)_1$ over any compact Lipschitz metric manifold $(N, d)$ has the property $\ED$. 
\end{example}

\begin{proof} 
Let $L_r = C_0(N)_r$ $(r \geq 0)$. 
Since $L_1$ : $\LD$, 
for the admissible tuple $(L_4, L_3, L_2, \emptyset, \emptyset) \in {\cal S}(L_1)$ 
there exist $\e \in (0,1)$ and an admissible deformation over $L_4$,   
$$\chi : \E^u_\ast(i_{L_2}, \e; L_2, L_1) \times [0,1] \lra \E^u_\ast(i_{L_2}, 1; L_2, L_1).$$

Given any $F \in {\cal C \cal F}(L_1)$ and $\alpha > 0$, take $\beta \geq 1$ such that $\e \beta > \alpha$ and $L_\beta \subset F$ and 
let $H = L_{4\beta}$.
The $\beta$-similarity transformation 
$$k_\beta : C_0(N, d) \to C_0(N, d), \ \ k_\beta(tx) = (\beta t)x$$ 
induces the $\beta$-similarity transformation 
$$\eta : \E^u_\ast(i_{L_2}, 1; L_2, L_1) \approx \E^u_\ast(i_{L_{2\beta}}, \beta; L_{2\beta}, L_\beta): \ \eta(f) =  k_\beta f \, k_{1/\beta}.$$ 

Since $\eta^{-1}(\E^u_\ast(i_{L_{2\beta}}, \alpha; L_{2\beta}, L_\beta)) \subset \E^u_\ast(i_{L_2}, \e; L_2, L_1)$,  
we obtain an admissible deformation over $H$, 
$$\psi_t = \eta \chi_t \eta^{-1} : \E^u_\ast(i_{L_{2\beta}}, \alpha; L_{2\beta}, L_\beta) \lra 
\E^u_\ast(i_{L_{2\beta}}, \beta; L_{2\beta}, L_\beta) \ \ (t \in [0,1]),$$
such that for each $(f, t) \in \E^u_\ast(i_{L_{2\beta}}, \alpha; L_{2\beta}, L_\beta) \times [0,1]$ 
\begin{itemize}  
\item[] (i) $\psi_t(f) = f$ \ on \ $L_{2\beta} - L_{3\beta}$  \ and \ (ii) $\psi_t(f)(L_{2\beta}) = f(L_{2\beta})$. 
\end{itemize}  

Finally the required admissible deformation over $H$ is defined by 
$$\phi : \E^u_\ast(i_F, \alpha; F, L_1) \times [0,1] \lra \E^u_\ast(i_F, \beta; F, L_1) : \hspace{5mm} 
\phi_t(f) = 
\left\{ \hspace{-1mm}  
\begin{array}[c]{ll} 
\psi_t(f|_{L_{2\beta}}) & \text{on} \ L_{2\beta}, \\[2mm] 
\ f & \text{on} \ F - L_{2\beta}. 
\end{array}\right.$$ 
\vskip -6mm 
\end{proof}


\subsection{End deformation theorem for uniform embeddings} \mbox{} 

Suppose $(M, d)$ is a metric $n$-manifold.
A proper product end of $(M, d)$ is a closed subset $L$ of $M$ such that ${\rm Fr}_M L$ is compact and $(L, d|_L)$ is a proper product end. 
Suppose $L$ is a proper product end of $(M, d)$. 
We say that the end $L$ is isolated if for any $c > 0$ there exists $F \in {\cal C \cal F}(L)$ with $d(M - L, F) > c$. 
This condition is equivalent to the condition that $d(M - L, L_r) \to \infty$ as $r \to \infty$. 
Let ${\cal C \cal F}(M, L)$ denote the collection of closed subsets $F$ of $M$ such that $F \cap L \in {\cal C \cal F}(L)$. 
The next lemma is a refinement of \cite[Lemma 2.6]{Ya}. 

\begin{lemma}\label{lem_p-isotopy} Suppose $(M, d)$ is a metric manifold and $N \subset F \subset M$. 
Assume that there exists a compact manifold $S$ and 
a homeomorphism $\theta : S \times [0,3] \approx N$ such that $\theta(S \times (0,3))$ is an open subset of $M$. 
Let $N_I :=\theta(S \times I)$ for $I \subset [0,3]$ and $N_a := N_{[0,a]}$ for $a \in [0,3]$.  
Then, there exists a strong deformation retraction 
$$\mbox{$\phi$ \ \ of \ \  $\E_\ast^u(F, M; N_1)$ \ \  onto \ \ $\E_\ast^u(F, M; N_2)$}$$
such that for each $(f, t) \in \E_\ast^u(F, M; N_1) \times [0,1]$ 
\begin{itemize}
\item[(i)\ ] $\phi_t(f) = f$ \ on $f^{-1}(M - N_{(0,3)}) - N_{(0,3)}$, 
\item[(ii)\,] $d(\phi_t(f), i_F) \leq d(f, i_F) + {\rm diam}\,N$, 
\item[(iii)] if $f = \id$ on $F \cap \partial M$, then $\phi_t(f) = \id$ on $F \cap \partial M$ \ and   
\item[(iv)] if $N \subset f(F)$, then $\phi_t(f)(F) = f(F)$. 
\end{itemize}
\end{lemma}

\begin{proof} 
One can construct a map $\xi : [0, 1] \times [0,1] \lra {\cal C}^u(M, M)$ which satisfies the following conditions (cf. \cite[Proof of Lemma 2.6]{Ya}): 
\begin{itemize}
\item[(a)] $\xi_0(s) = \id_M$, \hspace{5mm} 
\item[(b)] $\xi_t(s)(N) = N$, \ $\xi_t(s)(N_1) = N_{1+t}$ \ and \\
\hspace*{5mm} 
$\xi_t(s) = \id$ \ on \ $(M - N) \cup \theta(S \times (\{ 0 \} \cup [2+s, 3]))$ \hspace{5mm} ($(s,t) \in [0, 1] \times [0,1]$), 
\item[(c)] $\xi_t(s) \in {\cal H}^u(M)$ \ \ ($(s,t) \in [0,1] \times [0,1] - \{ (0,1) \}$). 
\end{itemize} 
 Choose a map $\alpha : \E_\ast^u(F, M; N_1) \to [0,1]$ with $\alpha^{-1}(0) = \E_\ast^u(F, M; N_2)$ and 
define a homotopy 
$$\begin{array}[t]{l}
\phi : \E_\ast^u(F, M; N_1) \times [0,1] \lra \E_\ast^u(F, M; N_1) \hspace{10mm} \text{by} \\[3mm] 
\hspace{10mm} 
\phi_t(f) = 
\left\{ \begin{array}[c]{@{\,}ll}
\xi_t(\alpha(f)) \, f \,(\xi_t(\alpha(f))|_F)^{-1} & (f \in \E_\ast^u(F, M; N_1) - \E_\ast^u(F, M; N_2),  \\[2mm] 
\ f & (f \in \E_\ast^u(F, M; N_2). 
\end{array} \right.
\end{array}$$

\vskip 1mm 
\noindent If $f \in \E_\ast^u(F, M; N_1) - \E_\ast^u(F, M; N_2)$, then 
$\xi_t(\alpha(f))^{-1}(N_1) \subset N_1$ and $f = \id$ on $N_1$, so that $\phi_t(f) = \id$ on $N_1$. 
The continuity of $\phi$ follows from Lemma~\ref{lemma_conti} 
applied to $P = \E_\ast^u(F, M; N_1) \times [0,1]$ and the maps 
$$\mbox{$\zeta : P \to {\cal C}(F, F)$ : \ $\zeta(f,t) = \xi_t(\alpha(f))|_F$ \ \ and \ \ 
$\eta : P \to {\cal C}(F, M)$ : \ $\eta(f, t) = \xi_t(\alpha(f)) f$.}$$ 
In fact, since $\xi_t(0)(N_2) = N_2$ and $\xi_t(0) = \id$ on $M - N_2$ by (b), it follows that 
$f \,\xi_t(0)|_F = \xi_t(0) f$ $((f,t) \in \E_\ast^u(F, M; N_2) \times [0,1])$ and so  
$$\phi_t(f) \, \zeta(f,t) = \eta(f,t) \hspace{5mm} ((f,t) \in P).$$ 

For each $f \in \E_\ast^u(F, M; N_1) - \E_\ast^u(F, M; N_2)$ we have 
$\xi_1(\alpha(f))^{-1}(N_2) = N_1$, so $\phi_1(f) = \id$ on $N_2$. 
These observations imply that $\phi$ is a strong deformation retraction of $\E_\ast^u(F, M; N_1)$ onto $\E_\ast^u(F, M; N_2)$.
Finally, the additional properties (i), (ii) and (iv) follow from the condition (b) $\xi_t(s) = \id$ on $M - N_{(0,3)}$.
This completes the proof. 
\end{proof}

\begin{lemma}\label{lem_ED_end} Suppose $(M, d)$ is a metric manifold and $L$ is an isolated proper product end of $(M,d)$. 
If $(L, d|_L) : \ED$, then the following holds: 
\begin{itemize}
\item[$(\#)$] For any $F \in {\cal C \cal F}(M, L)$, any $s > r > 1$ with $L_r \subset {\rm Int}_M (F \cap L)$ and any $\alpha > 0$ 
there exist $\beta > 0$ and an admissible deformation over $L_s$, \\[1mm] 
\hspace*{30mm} $\phi : \E^u_\ast(i_F, \alpha; F, (M, d)) \times [0,1] \lra \E^u_\ast(i_F, \beta; F, (M, d))$ \\[1mm] 
 such that for each $(f, t) \in \E^u_\ast(i_F, \alpha; F, (M, d)) \times [0,1]$ 
\begin{itemize}
\item[(i)\ ] $\phi_t(f) = f$ \ on $f^{-1}(M - {\rm Int}_M L_r) - {\rm Int}_M L_r \subset F$,  
\item[(ii)\,] if $L_r \subset f(F)$, then  $\phi_t(f)(F) = f(F)$ \ \ and \ \ {\rm (iii)} if $f = \id$ on $L_s$, then $\phi_t(f) = f$. 
\end{itemize}
\end{itemize}
\end{lemma}

\begin{proof} 
Since $L$ is isolated, we can find $u > s$ such that $O_\alpha(L_{u}, M) \subset L$. 
Since $(L, d|_L)$ : (ED), for $L_{u}$ and $\alpha$ we can find $\gamma > \alpha$, $v > u$ and an admissible deformation over $L_v$ \\[1mm] 
\hspace*{40mm} $\chi : \E^u_\ast(i_{L_{u}}, \alpha; L_{u}, (L, d|_L)) \times [0, 1] \lra \E^u_\ast(i_{L_{u}}, \gamma; L_{u}, (L, d|_L))$  \\[1mm] 
such that for each $(f,t) \in \E^u_\ast(i_{L_{u}}, \alpha; L_{u}, L) \times [0,1]$ 
\begin{itemize}
\item[] (i) $\chi_t(f) = f$ \ on ${\rm Fr}_L L_{u}$ \ \ and \ \ (ii) $\chi_t(f)(L_{u}) = f(L_{u})$. 
\end{itemize}
Define an admissible deformation over $L_v$,  
$$\psi : \E^u_\ast(i_F, \alpha; F, M) \times [0,1] \lra \E^u_\ast(i_F, \gamma; F, M) \ \ 
\text{by} \ \ 
\psi_t(f) = 
\left\{ \begin{array}[c]{ll}
\chi_t(f|_{L_{u}}) & \text{on} \ L_{u}, \\[2mm]
f & \text{on} \ F - {\rm Int}_M L_{u},  
\end{array} \right.$$
which satisfies the following conditions: \hspace{10mm} For each $(f,t) \in \E^u_\ast(i_F, \alpha; F, M) \times [0,1]$
\begin{itemize}
\item[(i)] $\psi_t(f) = f$ \ on $F - {\rm Int}_M L_{u}$, \ \ (ii) $\psi_t(f)(F) = f(F)$, \ \ (iii)  if $f = \id$ \ on $L_{u}$, then $\psi_t(f) = f$. 
\end{itemize}

Choose any $w > v$ and let $\beta := \gamma + {\rm diam}\,L_{[r, w]}$. 
We apply Lemma~\ref{lem_p-isotopy} to $L_{[r, w]} \subset F$ to obtain a strong deformation retraction 
$$\mbox{$\eta$ \ \ of \ \  $\E_\ast^u(F, M; L_{[v, w]})$ \ \  onto \ \ $\E_\ast^u(F, M; L_{[s, w]})$}$$
such that for each $(f, t) \in \E_\ast^u(F, M; L_{[v, w]}) \times [0,1]$ 
\begin{itemize}
\item[(i)\ ] if $f = \id$ on $F \cap \partial M$, then $\eta_t(f) = \id$ on $F \cap \partial M$,  
\item[(ii)\,] $d(\eta_t(f), i_F) \leq d(f, i_F) + {\rm diam}\,L_{[r, w]}$, 
 \item[(iii)] $\eta_t(f) = f$ \ on $f^{-1}(M - L_{(r, w)}) - L_{(r, w)}$ \ \ and  
\item[(iv)] if $L_{[r, w]} \subset f(F)$, then $\eta_t(f)(F) = f(F)$. 
\end{itemize}

The deformation $\eta$ restricts to an admissible deformation over $L_s$, \\[1mm] 
\hspace*{40mm} $\eta : \E^u_\ast(i_F, \gamma; F, M; L_v) \times [0,1] \lra \E^u_\ast(i_F, \beta; F, M;L_v)$ \\[1mm]
such that for each $(f,t) \in \E^u_\ast(i_F, \gamma; F, M; L_v)\times [0,1]$ 
\begin{itemize}
\item[(i)\ ] $\eta_t(f) = f$ \ on \ $f^{-1}(M - {\rm Int}_M L_r) - {\rm Int}_M L_r$,  
\item[(ii)\,] if $L_{[r,w]} \subset f(F)$, then $\eta_t(f)(F) = f(F)$ \ \ and 
\item[(iii)] if $f = \id$ on $L_s$, then $\eta_t(f) = f$. 
\end{itemize}

Finally, the required admissible deformation $\phi$ is defined by \\[1mm] 
\hspace*{10mm} $\phi : \E^u_\ast(i_F, \alpha; F, M) \times [0,1] \lra \E^u_\ast(i_F, \beta; F, M)$ : \ \  
$\phi_t(f) = 
\left\{ \hspace{-1mm} 
\begin{array}[c]{ll}
\psi_{2t}(f) & (t \in [0,1/2]),\\[2mm] 
\eta_{2t-1}\psi_1(f) & (t \in [1/2,1]). 
\end{array} \right.$
\vskip -4.5mm 
\end{proof}

The next remark is useful in the arguments below. 

\begin{remark}\label{rem_composition} 
Suppose $(M, d)$ is a metric manifold and $A \subset F \subset M$. Let $\E := \E^u_\ast(F, (M,d))_b$. 
Consider the following conditions for a map $\sigma : \E \to \E$ : 
\begin{itemize}
\item[(i)\ ] $\sigma(f) = f$ on $f^{-1}(M - A) - A$ for any $  f \in \E$. 
\item[(ii)\,] $\sigma(f)(F) = f(F)$ for any $f \in \E$ with $A \subset f(F)$. 
\item[(iii)] $\sigma(f) = \id$ on $F \cap \partial M$ for any $f \in \E$ with $f = \id$ on $F \cap \partial M$. 
\end{itemize} 

All of these conditions are preserved under composition 
(i.e., if both $\sigma, \tau : \E \to \E$ satisfy the condition (i) ((ii), (iii) respectively), 
then so does the composition  $\tau \sigma : \E \to \E$). 

The assertion for (i) is verified as follows: 
For $f \in \E$, let $X_f := f^{-1}(M - A) - A$. 
It is seen that $f, g \in \E$ and $g = f$ on $X_f$, then $X_f \subset X_g$. 
Hence, if both $\sigma$, $\tau$ satisfy the condition (i), then for any $f \in \E$ 
it follows that $\tau(\sigma(f)) = \sigma(f)$ on $X_{\sigma(f)}$, $\sigma(f) = f$ on $X_f$ and $X_f \subset X_{\sigma(f)}$, 
so that $(\tau \sigma)(f) = f$ on $X_f$. 
\end{remark}

\begin{proposition}\label{prop_ED_end}
Suppose $(M, d)$ is a metric manifold and $L$ is an isolated proper product end of $(M, d)$. 
If $(L, d|_L) : \ED$, then for any $F \in {\cal C \cal F}(M, L)$ and any $s > r > 1$ with $L_r \subset {\rm Int}_M (F \cap L)$ 
there exists a strong deformation retraction \\
\hspace*{40mm} $\phi$ \ of \ $\E^u_\ast(F, (M, d))_b$ \ onto \ $\E^u_\ast(F, (M, d); L_s)_b$ \\
which satisfies the following conditions:   
\begin{enumerate} 
\item For each $(f,t) \in \E^u_\ast(F, (M, d))_b \times [0,1]$ 
\begin{itemize}
\item[(i)\ ] $\phi_t(f) = f$ \ on $f^{-1}(M - {\rm Int}_M L_r) - {\rm Int}_M L_r$,  
\item[(ii)\,] if $L_r \subset f(F)$, then $\phi_t(f)(F) = f(F)$ \ \ and 
\item[(iii)] if $f = \id$ on $F \cap \partial M$, then $\phi_t(f) = \id$ on $F \cap \partial M$. 
\end{itemize}
\item For any $\alpha > 0$ there exists $\beta > 0$ such that \\
\hspace*{10mm} $\phi(\E^u_\ast(i_F, \alpha; F, (M, d)) \times [0,1]) \subset \E^u_\ast(i_F, \beta; F, (M, d))$. 
\end{enumerate}  
\end{proposition}

\begin{proof} Given $F, r, s$. For notational simplicity, we put $\E = \E^u_\ast(F, M)_b$ and 
$\E(\alpha) = \E^u_\ast(i_F, \alpha; F, M)$ $(\alpha > 0)$.  
The repeated applications of Lemma~\ref{lem_ED_end} yield a sequence 
of real numbers $\alpha_i > 0$ $(i \in \IN)$ and admissible deformations over $L_s$, 
$$\chi^i : \E(\alpha_i+2) \times [0,1] \lra \E(\alpha_{i+1}) \ \ (i \in \IN),$$ 
which satisfy the following conditions:  \hspace{5mm}  For each $i \in \IN$ 

\begin{enumerate} 
\item[(3)] $\alpha_i + 2 < \alpha_{i+1}$,  
\item[(4)] for each $(f, t) \in \E(\alpha_i+2) \times [0,1]$ 
\begin{itemize}
\item[(i)\ ] $\chi^i_t(f) = f$ \ on $f^{-1}(M - {\rm Int}_M L_r) - {\rm Int}_M L_r$,   
\item[(ii)\,] if $L_r \subset f(F)$, then $\chi^i_t(f)(F) = f(F)$ \ \ and \ \ 
(iii) if $f = \id$ on $L_s$, then $\chi^i_t(f) = f$. 
\end{itemize} 
\end{enumerate}  

For each $i \in \IN$ take a map $\lambda_i : \E(\alpha_i+2) \lra [0,1]$ such that 
$\lambda_i =  1$ on $\E(\alpha_i)$ and $\lambda_i =  0$ on $\E(\alpha_i+2) - \E(\alpha_i+1)$ 
and define a homotopy $\eta^i$ on $\E$ by 
$$\eta^i : \E \times [0,1] \lra \E : \hspace{2mm} 
\eta_t^i(f)  =  
\left\{ \hspace{-0.5mm} 
\begin{array}[c]{ll}
(\chi^i)_{\lambda_i(f) t}(f) & (f \in \E(\alpha_i+2)), \\[2mm] 
f & (f \in \E - \E(\alpha_i+1)). 
\end{array} \right.$$
It satisfies the following conditions: 
\begin{itemize}
\item[(5)]
\begin{itemize}
\item[(i)\ ] for each  $(f, t) \in \E \times [0,1]$ \\
\hspace*{1mm} (a) $\eta^i_0(f) = f$, \hspace{5mm} 
(b) $\eta^i_t(f) = f$ \ on \ $f^{-1}(M - {\rm Int}_M L_r) - {\rm Int}_M L_r $, \\  
\hspace*{1mm} (c) if $L_r \subset f(F)$, then $\eta^i_t(f)(F) = f(F)$, \\
\hspace*{1mm} (d) if $f = \id$ on $F \cap \partial M$, then $\eta^i_t(f) = \id$ on $F \cap \partial M$,  
\item[(ii)\,] 
(a) $\eta^i_t(f) = f$ $(t \in [0,1])$ \ if $f \in \E^u_\ast(F, M; L_s)_b$ or $f \not\in \E(\alpha_i+1))$, \\
(b) $\eta^i_t(\E(\alpha_i+2)) \subset \E(\alpha_{i+1})$ $(t \in [0,1])$ 

\item[(iii)] $\eta^i_1(\E(\alpha_i)) \subset \E^u_\ast(F, M; L_s)_b$ 
\item[(iv)] $\eta^j_t(\E(\alpha_i)) \subset \E(\alpha_i)$ \ \ $(j \leq i-1$, $t \in [0,1]$) 
\item[(v)\,] 
(a) $\eta^i_1\eta^{i-1}_1 \dots \eta^1_1(\E(\alpha_i)) \subset \eta^i_1(\E(\alpha_i)) \subset \E^u_\ast(F, M; L_s)_b$, \\ 
(b) $\eta^j_t \eta^{j-1}_1 \cdots \eta^i_1 \dots \eta^1_1(f) = \eta^i_1 \dots \eta^1_1(f)$ \ \ ($f \in \E(\alpha_i)$, $j \geq i+1$, $t \in [0,1]$).

\end{itemize}
\end{itemize} 

Replacing $[0,1]$ by $[0, \infty]$, the required homotopy $\phi : \E \times [0,\infty] \to \E$ is defined by 
$$\phi_t(f) = \left\{ \begin{array}[c]{@{\ }ll}
\eta^j_{t -j+1} \eta^{j-1}_1 \cdots \eta^1_1(f) & (t \in [j-1, j], j \in \IN), \\[2mm] 
\ds \lim_{j \to \infty}\eta^j_1 \cdots \eta^1_1(f) & (t = \infty). 
\end{array} \right.$$
\vskip 1mm 
The homotopy $\phi_t$ is well-defined for $t \in [0, \infty)$ by (5)(i)(a) and for $t = \infty$ by (5)(v)(b). 
The continuity of $\phi$ follows from 
$$\phi_t(f) = \eta^i_1 \dots \eta^1_1(f) \ \ ((f, t) \in \E(\alpha_i) \times [i, \infty]).$$ 
By (5)(i)(a), (ii)(a), (v)(a) the homotopy $\phi$ is a strong deformation retraction of $\E$ onto $\E^u_\ast(F, M, L_s)_b$. 
It remains to verify the conditions (1) and (2). 

(1) Consider the following condition $(\ast)$  for a map $\sigma : \E \to \E$ : 
\begin{itemize}
\item[$(\ast)$]
\begin{itemize}
\item[(i)\ ] $\sigma(f) = f$ on $X_f \equiv f^{-1}(M - {\rm Int}_M L_r) - {\rm Int}_M L_r$ for any $  f \in \E$. 
\item[(ii)\,] $\sigma(f)(F) = f(F)$ for any $f \in \E$ with $L_r \subset f(F)$. 
\item[(iii)] $\sigma(f) = \id$ on $F \cap \partial M$ for any $f \in \E$ with $f = \id$ on $F \cap \partial M$. 
\end{itemize} 
\end{itemize} 
By Remark~\ref{rem_composition}, the condition $(\ast)$ is preserved under composition. 
We have to show that the map $\phi_t$ satisfies this condition $(\ast)$ for any $t \in [0, \infty]$. 
By (5)(i) $\eta^i_t$ satisfies the condition $(\ast)$ for each $i \in \IN$ and $t \in [0,1]$. 
If $t \in [j-1, j]$ ($j \in \IN$), then the map $\phi_t = \eta^j_{t -j+1} \eta^{j-1}_1 \cdots \eta^1_1$ satisfies the condition $(\ast)$.  
For $t = \infty$, if $f \in \E(\alpha_i)$, then $\phi_\infty(f) = \eta^i_1 \dots \eta^1_1(f)$ and 
since $\eta^i_1 \dots \eta^1_1$ satisfies the condition $(\ast)$, 
it follows that 
(i) $\phi_\infty(f) = f$ on $X_f$, 
(ii) $\phi_\infty(f)(F) = f(F)$ if $L_r \subset f(F)$ and 
(iii) $\phi_\infty(f) = \id$ on $F \cap \partial M$ if $f = \id$ on $F \cap \partial M$. 
This implies the conclusion. 

(2) From (5) it follows that $\phi_t(\E(\alpha_i)) \subset \E(\alpha_{i+1})$ $(i \in \IN, t \in [0,\infty])$. 
\end{proof}

\begin{theorem}\label{thm_ED_sdr}
Suppose $(M, d)$ is a metric manifold and $L(1), \cdots, L(m) \subset M$ are pairwise disjoint isolated  proper product ends of $M$. 
If $(L(i), d|_{L(i)}) : \ED$ for each $i =1, \cdots, m$, then for any $F \subset M$ with $F \in {\cal C \cal F}(M, L(i))$ $(i =1, \cdots, m)$ and any $s_i > r_i > 1$ with $L(i)_{r_i} \subset {\rm Int}_M (F \cap L(i))$ $(i =1, \cdots, m)$ 
there exists a strong deformation retraction \\[1mm] 
\hspace*{40mm} $\phi$ \ of \ $\E^u_\ast(F, (M, d))_b$ \ onto \ $\E^u_\ast(F, (M, d); \cup_{i=1}^m L(i)_{s_i})_b$ \\[1mm]
such that 
\begin{enumerate} 
\item for each $(f, t) \in \E^u_\ast(F, (M, d))_b \times [0,1]$ 
\begin{itemize} 
\item[(i)\ ] $\phi_t(f) = f$ \ 
on $f^{-1}(M - {\rm Int}_M (\cup_{i=1}^m L(i)_{r_i})) - {\rm Int}_M (\cup_{i=1}^m L(i)_{r_i})$,   
\item[(ii)\,] if $\cup_{i=1}^m L(i)_{r_i} \subset f(F)$, then $\phi_t(f)(F) = f(F)$, 
\item[(iii)]  if $f = \id$ on $F \cap \partial M$, then $\phi_t(f) = \id$ on $F \cap \partial M$. 
\end{itemize} 
\vskip 1mm 
\item for each $  \alpha > 0$ there exists $ \beta > 0$ such that  \\ 
\hspace*{20mm} $\phi(\E^u_\ast(i_F, \alpha; F, (M, d)) \times [0,1]) \subset \E^u_\ast(i_F, \beta; F, (M, d))$. 
\end{enumerate} 
\end{theorem}

\begin{proof} 
Let $\E = \E^u_\ast(F, M)_b$ and $\E(\alpha) = \E^u_\ast(i_F, \alpha; F, M)$ $(\alpha > 0)$. 
By Proposition~\ref{prop_ED_end}, 
for each $i =1, \cdots, m$
there exists a strong deformation retraction 
$$\mbox{$\phi^i$ \ of \ $\E$ \ onto \ $\E^u_\ast(F, M; L(i)_{s_i})_b$}$$
which satisfies the following conditions: 
\begin{enumerate} 
\item[(3)] For each $(f, t) \in \E \times [0,1]$ 
\begin{itemize}
\item[(i)\ ] $\phi^i_t(f) = f$ \ on $f^{-1}(M - {\rm Int}_M L(i)_{r_i}) - {\rm Int}_M L(i)_{r_i}$,  
\item[(ii)\,] if $L(i)_{r_i} \subset f(F)$, then $\phi^i_t(f)(F) = f(F)$, 
\item[(iii)] if $f = \id$ on $F \cap \partial M$, then $\phi^i_t(f) = \id$ on $F \cap \partial M$. 
\end{itemize}
\item[(4)] For any $\alpha > 0$ there exists $\beta > 0$ such that 
$\phi^i(\E(\alpha) \times [0,1]) \subset \E(\beta)$. 
\end{enumerate}  

Then the composition \ $$\phi_t = \phi_t^m \cdots  \phi_t^1 : \E \lra \E \ \ (t \in [0, 1])$$ 
is a strong deformation retraction of $\E$ onto $\E^u_\ast(F, M; \cup_{i=1}^m L(i)_{s_i})_b$, that is, 
\begin{enumerate} 
\item[(5)] for each $  f \in \E$ \hspace{4mm} 
\begin{tabular}[t]{c@{\ \,}l}
(i) & $\phi_0(f) = f$, \hspace{4mm} (ii) $\phi_1(f) = \id$ on $\cup_{i=1}^m L(i)_{s_i}$, \\[2mm] 
(iii) & if $f = \id$ on $\cup_{i=1}^m L(i)_{s_i}$, then $\phi_t(f) = f$ $(t \in [0,1])$. 
\end{tabular}  
\end{enumerate} 
\vspace*{1mm} 
The condition (5)(ii) follows from the next assertion: 
\begin{enumerate}
\item[]
\begin{itemize}
\item[(ii)$_k$] $\phi_1^k \cdots  \phi_1^1(f) = \id$ \ on \ $\cup_{i=1}^k L(i)_{s_i}$ \hspace{5mm} $(k = 1, \cdots, m)$.   
\end{itemize} 
\end{enumerate}
This is verified inductively based upon (3)(i) and $\phi^i_1(f) = \id$ on $L(i)_{s_i}$ $(i=1, \cdots, m)$.

The condition (1) follows from (3) and Remark~\ref{rem_composition}. 
The condition (2) follows from (4). 
In fact, for any $\alpha_0 > 0$, we can find $\alpha_i > 0$ $(i=1, \cdots, m)$ such that 
$$\phi^i(\E(\alpha_{i-1}) \times [0,1]) \subset \E(\alpha_i) \hspace{4mm} (i=1, \cdots, m).$$ 
Then we have $\phi(\E(\alpha_0) \times [0,1]) \subset \E(\alpha_m)$. 
This completes the proof. 
\end{proof}

\begin{proof}[\bf Proof of Theorem~\ref{thm_ED_sdr_H}]
 By Theorem~\ref{thm_ED_sdr}, 
for $F = M$ and $(r_i, s_i) = (2,3)$ $(i =1, \cdots, m)$ 
we obtain a strong deformation retraction of $\E^u_\ast(M, M)_b$ onto $\E^u_\ast(M, M; L_3)_b$, 
which restricts to the required strong deformation retraction 
$\phi$ of ${\cal H}^u_b(M)$ onto ${\cal H}^u_b(M; L_3)$. 
\end{proof}


\section{Uniform isotopies}

\subsection{Uniform isotopies} \mbox{}

Suppose $X$ is a topological space. 
An isotopy on $X$ is a homeomorphism $H \in {\cal H}(X \times [0,1])$ which preserves $[0,1]$-factor 
(i.e., $H(x,t) = (H_t(x), t)$ $((x,t) \in X \times [0,1])$). 
By ${\cal H}(X \times [0,1])^I$ we denote the subgroup of ${\cal H}(X \times [0,1])$ consisting of all isotopies on $X$. 
If $H$ is an isotopy on $X$, then $H_t \in {\cal H}(X)$ for each $t \in [0,1]$ and we obtain the associated function 
$\widehat{H} : [0,1] \to {\cal H}(X)$ : $\widehat{H}(t) = H_t$. 
Usually an isotopy $H$ on $X$ is regarded as a sort of path in ${\cal H}(X)$ and 
described as the family $\{ H_t \}_{t \in [0,1]}$. 
Note that $(HK)_t = H_tK_t$ and $(H^{-1})_t = (H_t)^{-1}$ for any $H, K \in {\cal H}(X \times [0,1])^I$. 
For a subset ${\cal S}$ of ${\cal H}(X)$ and $f, g \in {\cal S}$
an  isotopy from $f$ to $g$ in ${\cal S}$ means an isotopy $H$ on $X$ such that 
$H_0 = f$, $H_1 = g$ and $H_t \in {\cal S}$ $(t \in [0,1])$. 

Suppose $(X, d)$ is a metric space. 
The product space $X \times [0,1]$ is given the metric $\tilde{d}$ defined by 
$$\mbox{$\tilde{d}((x,t), (y,s)) = d(x,y) + |t-s|$ \ \ ($(x,t), (y,s) \in X \times [0,1])$.}$$ 
The metric $\tilde{d}$ has the following basic properties. 

\begin{lemma}\label{lem_A_tilde_d} {\rm (1)} $(X \times [0,1], \tau_{\tilde{d}}) = (X, \tau_d) \times [0,1]$ $($Tychonoff product$)$.
\begin{enumerate}
\item[(2)] If $\rho$ is a metric on $X \times [0,1]$ uniformly equivalent to $\tilde{d}$, then 
\begin{itemize}
\item[(i)\,] $\tau_{\rho} = \tau_{\tilde{d}}$ on $X \times [0,1]$ 
\item[(ii)] $(a)$ ${\cal H}(X \times [0,1], \rho)_u = {\cal H}(X \times [0,1], \tilde{d})_u$ and \\
$(b)$ the sup-metrics on this space defined by $\rho$ and $\tilde{d}$ are uniformly equivalent. 
\end{itemize}

\item[(3)] Any $H \in {\cal H}(X \times [0,1], \tilde{d})^I$ induces the associated function 
$\widehat{H} : [0,1] \to {\cal H}(X, d)$. 
We have 
$\ds \tilde{d}(H, K) = d(\widehat{H}, \widehat{K}) \ \big( = \sup_{t \in [0,1]} d(H_t, K_t) \big)$ \ for any $H, K \in {\cal H}(X \times [0,1], \tilde{d})^I$.  
\hfill $\square$ 
\end{enumerate}
\end{lemma}

\begin{defn} An isotopy $H$ on $(X, \tau_d)$ is said to be 
\vspace*{-1mm}
\begin{itemize}
\item[(i)\,] a (bounded) uniform isotopy on $(X, d)$ if $H \in {\cal H}^u_{(b)}(X \times [0,1], \tilde{d})$, 
\item[(ii)] a level-wise bounded (uniform) isotopy on $(X, d)$ if $H_t \in {\cal H}^{(u)}_{b}(X, d)$ for each $t \in [0,1]$. 
\end{itemize}
\end{defn}

Let ${\cal H}^u_{(b)}(X \times [0,1], \tilde{d})^I$ and ${\cal H}^{lb(u)}(X \times [0,1], \tilde{d})^I$ 
denote the subgroups of ${\cal H}(X \times [0,1])^I$ consisting of (bounded) uniform isotopies and level-wise bounded (uniform) isotopies on $(X, d)$, 
endowed with the sup-metric and the uniform topology defined by $\tilde{d}$.  

Uniform isotopies are reinterpreted correctly as continuous paths in ${\cal H}^u(X,d)$. 

\begin{lemma}\label{lem_A_unform_isotopy}
{\rm (1)} If $H$ is a $($bounded\,$)$ uniform isotopy on $(X, d)$, then {\rm (i)} $H$ is a level-wise $($bounded\,$)$ uniform isotopy on $(X, d)$ and 
{\rm (ii)} the function $\widehat{H} : [0,1] \to {\cal H}^u_{(b)}(X,d)$, $\widehat{H}(t) = H_t$ \ is continuous. 

{\rm (2)} Any map $h : [0,1] \to {\cal H}^u_{(b)}(X,d)$ determines a $($bounded\,$)$ uniform isotopy $\tilde{h}$ on $(X,d)$ by 
$$\tilde{h}(x,t) = (h(t)(x), t) \hspace{5mm} ((x,t) \in X \times [0,1]).$$ 

{\rm (3)} The correspondence given in {\rm (1)} and {\rm (2)} defines a natural isometry 
$$\eta : \big({\cal H}^u_{(b)}(X \times [0,1], \tilde{d})^I,  \tilde{d}\big) \lra \big({\cal C}([0,1], ({\cal H}^u_{(b)}(X,d), d))_u, d\big) : 
\hspace{5mm} \eta(H) = \widehat{H}, \ \ \eta^{-1}(h) = \tilde{h}.$$
\end{lemma}

\begin{proof}
(1)(ii) Since $H$ is uniformly continuous, for any $\e > 0$ there exists $\delta > 0$ such that 
if $\tilde{d}((x,t), (y,s)) < \delta$, then $\tilde{d}(H(x,t), H(y,s)) < \e$. 
It is seen that if $|t - s| < \delta$, then $d(H_t, H_s) \leq \e$. 

(2) By the compactness of $[0,1]$ it follows that 
$h$ is uniformly continuous and the family $\{ h(t) \}_{t \in [0,1]}$ is equi-continuous. 
From these observations it follows that $\tilde{h}$ is uniformly continuous. 

Since ${\cal H}^u(X,d)$ is a topological group, 
we have a map $k : [0,1] \to {\cal H}^u(X,d)$ : $k(t) = h(t)^{-1}$ and the associated map 
$$\tilde{k} : X \times [0,1] \lra X \times [0,1]: \ \ \tilde{k}(x,t) = (k(t)(x), t).$$ 
is also uniformly continuous.
Since $\tilde{k}\tilde{h} = \tilde{h}\tilde{k} = \id_{X \times [0,1]}$, 
it follows that $\tilde{h}$ is a uniform isotopy on $(X, d)$. 
In the bounded case, since $h$ is continuous, by Lemma~\ref{lem_A_tilde_d}(3) we have 
$$\ds \tilde{d}(\tilde{h}, \id_{X \times [0,1]}) = d(h, \id_X) = \max_{t \in [0,1]} d(h(t), \id_X) < \infty.$$ 

(3) The assertion follows from Lemma~\ref{lem_A_tilde_d}(3). 
\end{proof}

\begin{remark}\label{rem_A_exp-law} 
Suppose $Z$ is a topological space, $K$ is a compact topological space and $(Y, \rho)$ is a metric space. 
The exponential law for function spaces 
induces a natural isometry
$$\chi : \big({\cal C}(Z \times K, (Y, \rho)), \rho)
\begin{array}[b]{c}
\text{\footnotesize $\cong$} \\[-3mm] 
\longrightarrow 
\end{array}
\big({\cal C}\big(Z, \big({\cal C}(K, (Y,\rho))_u, \rho\big)\big), \rho\big),$$
where the metric $\rho$ on $Y$ induces 
(i) the sup-metric $\rho$ on ${\cal C}(Z \times K, (Y, \rho))$ and 
(ii) the sup-metric $\rho$ on ${\cal C}(K, (Y,\rho))_u$, which further induces the sup-metric $\rho$ on ${\cal C}(Z, {\cal C}(K, (Y,\rho))_u)$. 
Any map $\phi \in {\cal C}(Z \times K, (Y,\rho))$ and the corresponding map $\Phi = \chi(\phi)$ are related by \ $\Phi(z)(k) = \phi(z, k)$ \ $((z,k) \in Z \times K)$.
\end{remark}

By Lemma~\ref{lem_A_unform_isotopy}(3) and Remark~\ref{rem_A_exp-law} 
we obtain natural isometries between function spaces related to uniform isotopies. 

\begin{proposition}\label{prop_exp_low} 
For any topological space $Z$ there exist natural isometries 
$$\begin{array}[c]{rcl}
({\cal C}(Z \times [0,1], ({\cal H}^u_{(b)}(X,d), d)), d)
& \begin{array}[c]{@{}c@{}}
\text{\footnotesize $\chi$} \\[-2mm] 
\longrightarrow \\[-2mm] 
\cong 
\end{array} & 
\big({\cal C}\big(Z, ({\cal C}([0,1], ({\cal H}^u_{(b)}(X,d),d))_u, d)\big), d\big) \\[2mm] 
& \begin{array}[c]{@{}c@{}}
\text{\footnotesize $\eta_{\#}$} \\[-2mm] 
\longleftarrow \\[-2mm]
\cong 
\end{array} & 
\big({\cal C}\big(Z, \big({\cal H}^u_{(b)}(X \times [0,1], \tilde{d})^I, \tilde{d}\big)\big), \tilde{d} \big).
\end{array}$$   
For any map $\phi \in {\cal C}(Z \times [0,1], {\cal H}^u_{(b)}(X,d))$, 
the corresponding maps 
$\psi \in {\cal C}(Z, {\cal C}([0,1], ({\cal H}^u_{(b)}(X,d),d))_u)$ and   
$\Phi \in {\cal C}(Z, {\cal H}^u_{(b)}(X \times [0,1], \tilde{d})^I)$ are related by \ \ 
$$\mbox{$\chi(\phi) = \psi = \eta \Phi$ \ \ and \ \ $\phi(z,t) = \psi(z)(t) = \Phi(z)_t$ \ \ $((z, t) \in Z \times [0,1])$.}$$
\vskip -8mm \hfill $\square$
\end{proposition}


\subsection{Local contractibility of ${\cal H}^u(X, d)$} \mbox{}

Since the space ${\cal H}^u(X, d)$ is not necessarily locally contractible in the usual sense, 
in the literature some authors introduced another weaker notion of local contractibility of ${\cal H}^u(X, d)$ 
in order to obtain some affirmative results (cf. \cite[Section 5.6]{Ru}). 
In this subsection we clarify the relationship between this weaker notion and the standard notion of 
local contractibility of ${\cal H}^u(X, d)$ as a topological space. 
For this purpose, first we recall the standard notions on local contractibility of topological spaces. 
Suppose $Z$ is any topological space and $C$ is a subset of $Z$. 
We say that 
\begin{itemize}
\item[(i)\ ] $C$ is contractible in $Z$ if there exists a homotopy 
$\phi : C \times [0,1] \to Z$ such that $\phi_0(z) = z$ $(z \in C)$ and $\phi_1$ is a constant map 
(i.e., $\phi_1(z) \equiv z_0$ for some fixed point $z_0$ in $Z$), 
\item[(ii)\,] 
\begin{itemize}
\item[(a)] $Z$ is locally contractible 
if for each point $z \in Z$ and any neighborhood $V$ of $z$ in $Z$ there exists a neighborhood $U$ of $z$ in $V$ such that 
$U$ is contractible in $V$, 
\item[(b)] $Z$ is weakly locally contractible 
if each point of $Z$ has a neighborhood contractible in $Z$, 
\end{itemize}

\item[(iii)] $Z$ is contractible if it is contractible in itself. 
\end{itemize}
For a topological group $G$ it is shown that  
$G$ is locally contractible if and only if some neighborhood of the unit element of $G$ is contractible in $G$. 

Proposition~\ref{prop_exp_low} leads us to an interpretation of local contractibility of ${\cal H}^u_b(X,d)$ in terms of uniform isotopies.  
Note that ${\cal H}^u_b(X,d)$ is an open and closed subgroup of ${\cal H}^u(X,d)$. 

\begin{proposition}\label{prop_A_contraction} 
Suppose ${\cal A} \subset {\cal B} \subset {\cal H}^u(X,d)$ and $h_0 \in {\cal A}$. 
Then ${\cal A}$ is contractible in ${\cal B}$ if and only if 
there exists a map 
$\Phi : {\cal A} \to {\cal H}^u(X \times [0,1], \tilde{d})^I$ 
such that  for each $h \in {\cal A}$ the image $\Phi(h)$ is a uniform isotopy from $h$ to $h_0$ in ${\cal B}$. 
\end{proposition}

\begin{corollary}\label{cor_LC_uniform isotopy} 
{\rm (1)} ${\cal H}^u(X,d)$ is locally contractible if and only if 
there exists a neighborhood ${\cal U}$ of  
$\id_X$ in ${\cal H}^u_b(X, d)$ and a map 
$\Phi : \U \to {\cal H}^u(X \times [0,1], \tilde{d})^I$ 
such that for each $h \in {\cal U}$ the image $\Phi(h)$ is a uniform isotopy from $h$ to $\id_X$. 

{\rm (2)} ${\cal H}^u_b(X,d)$ is contractible if and only if 
there exists a map $\Phi : {\cal H}^u_b(X, d) \to {\cal H}^u(X \times [0,1], \tilde{d})^I$ 
such that for each $h \in {\cal H}^u_b(X, d)$ the image $\Phi(h)$ is a uniform isotopy from $h$ to $\id_X$. 
\end{corollary}

\begin{remark}\label{rem_lbu->u} In Corollary~\ref{cor_LC_uniform isotopy} (1), (2) we have ${\rm Im}\,\Phi \subset  {\cal H}^u_b(X \times [0,1], \tilde{d})^I$. 
\end{remark}

In \cite[Section 5.6]{Ru} for any kind of topology on function spaces (Whitney, uniform and compact-open topologies, etc), 
a weaker notion of local contractibility of homeomorphism groups is formulated 
based on continuous selections of isotopies to the identity map. 
As for the uniform topology this formulation 
just means replacing uniform isotopies by level-wise uniform isotopies in Corollary~\ref{cor_LC_uniform isotopy}. 

\begin{defn}\label{def_locally contractible*}
We say that ${\cal H}^u(X, d)$ is locally contractible* 
if there exist a neighborhood ${\cal U}$ of $\id_X$ in ${\cal H}^u_b(X, d)$ and a map 
$$\Phi : \U \lra {\cal H}^{lbu}(X \times [0,1], \tilde{d})^I$$ 
such that for each $h \in {\cal U}$ the image $\Phi(h)$ is a level-wise bounded uniform isotopy from $h$ to $\id_X$. 
The map $\Phi$ is called a local contraction* of ${\cal H}^u_b(X, d)$. 
If we can take ${\cal U} = {\cal H}^u_b(X, d)$, 
then we say that ${\cal H}^u_b(X, d)$ is contractible* and 
call the map $\Phi$ a contraction* of ${\cal H}^u_b(X, d)$. 
\end{defn}

To avoid any ambiguity, we should use the precise terminology as ``(locally) contractible by means of level-wise uniform isotopies'' 
instead of the temporary phrase ``(locally) contractible*''. 
But the former is too long for our treatment in this article. 
Note that in \cite[Section 5.6]{Ru} this weaker notion is also described by the same terminology ``(locally) contractible''. 

The next corollary is a direct consequence of Corollary~\ref{cor_LC_uniform isotopy}.  

\begin{corollary}\label{cor_LC_uniform isotopy_2} 
If ${\cal H}^u_b(X, d)$ is $($locally$)$ contractible, then ${\cal H}^u_b(X, d)$ is $($locally$)$ contractible*. 
\end{corollary}


\section{Alexander isotopies in $\kappa$-cones} 

In this section we discuss basic properties of Alexander isotopies in $\kappa$-cones $(\kappa \leq 0)$ (cf.\,\cite{Ki}).  
Throughout this section we consider the $\kappa$-cone $C_\kappa(X, d) = (C(X), \tilde{d}_\kappa)$ over a compact metric space $(X, d)$.  
The metric $\tilde{d}_\kappa$ on $C_\kappa(X, d)$ induces the sup-metric $\tilde{d}_\kappa$ on ${\cal H}(C_\kappa(X, d))$ 
and the metric $\tilde{\tilde{d}}_\kappa$ on $C_\kappa(X, d) \times [0,1]$. We follow the convention $0 \cdot \infty = 0$.

\subsection{General properties} \mbox{} 

The $\kappa$-cone $C_\kappa(X, d)$ admits the radial transformations 
$$\mbox{$\theta_t \in {\cal H}(C_\kappa(X, d)) : \ \ \theta_t(u) = tu$ \hspace{5mm} $(t \in (0, \infty))$.}$$ 

\begin{defn}\label{defn_Alexander's trick} (Alexander's trick) \ 
Each $h \in {\cal H}(C_\kappa(X, d))$ determines the following functions: 
\vspace{1mm} 
\begin{enumerate}
\item $h_t \in {\cal H}(C_\kappa(X, d))$ \ $(t \in [0,\infty))$ : \ \ 
$h_t = 
\left\{ \hspace{-1mm} \begin{array}[c]{ll}
\theta_t h (\theta_{1/t}) & (t \in (0, \infty)), \\[2mm]
\ \ \id & (t = 0).   
\end{array}\right. $
\vspace{1mm} 
\item $\eta_h : [0, \infty) \lra {\cal H}(C_\kappa(X, d)) : \ \  \eta_h(t) = h_t$ 
\item $\Phi(h) : C_\kappa(X, d) \times [0,1] \to C_\kappa(X, d) \times [0,1]$ : \ \ 
$\Phi(h)(u,t) = (h_t(u), t) = 
\left\{ \hspace{-1mm} \begin{array}[c]{ll}
\ds \big(th(\mbox{$\frac{1}{t}$} u), t\big) & (t \in (0, 1]), \\[2mm]
\ (u,0) & (t = 0).
\end{array}\right.$
\end{enumerate}
\end{defn} 

The function $\Phi(h)$ is always continuous on the open subset $C_\kappa(X, d) \times (0,1]$. 
However, it is not necessarily continuous on a point in $C_\kappa(X, d) \times \{ 0 \}$ (for instance, consider the rotation in $\IR^2$).

Recall the function $\lambda_\kappa$ defined in Remark~\ref{rem_kappa-cone} in Section 2.5. 

\begin{lemma}\label{lem_theta_t}
{\rm (1)} For any $x,y \in X$ and $r, s,t \in [0, \infty)$ 
\begin{enumerate}
\item[] 
\begin{itemize}
\item[(i)\,] $\lambda_\kappa (\tilde{d}_\kappa (tsx, tsy)) = \lambda_\kappa(2ts) \sin \frac{1}{\,2\,} d_\pi(x, y)
= \frac{\lambda_\kappa(2ts)}{\lambda_\kappa(2s)} \lambda_\kappa (\tilde{d}_\kappa (sx, sy)),$ 
\item[(ii)] $\tilde{d}_\kappa(trx, tsx) = t \tilde{d}_\kappa(rx, sx) = t|r-s|$. 
\end{itemize}
\item[(2)]  For any $t \in [0,1]$ 
\begin{itemize}
\item[(i)\,] $\lambda_\kappa(ts) \leq t\lambda_\kappa(s)$ \ $(s \in [0, \infty))$, 
\item[(ii)] $\lambda_\kappa(\tilde{d}_\kappa (tu, tv)) \leq t \lambda_\kappa(\tilde{d}_\kappa (u, v))$, \ 
$\tilde{d}_\kappa (tu, tv) \leq \tilde{d}_\kappa (u, v)$ \ \ $(u, v \in C_\kappa(X, d))$. 
\end{itemize}
\end{enumerate}
\end{lemma} 

\begin{proof}
(2)(ii) $\lambda_\kappa^2  (\tilde{d}_\kappa (trx, tsy)) 
= \lambda_\kappa^2 (t(r-s)) + \lambda_\kappa(2tr) \lambda_\kappa(2ts) \sin^2 \mbox{$\frac{1}{\,2\,}$} d_\pi(x, y) 
\leq t^2 \lambda_\kappa^2  (\tilde{d}_\kappa (rx, sy)).$
\end{proof}

\begin{remark}\label{rmk_B} (1) When $\kappa = 0$, the map $\theta_t$ is $t$-similar for $t \in (0, \infty)$.  

(2) When $\kappa < 0$, the map $\theta_t$ is 1-Lipschitz for $t \in (0,1]$. 
Furthermore, if $(X, d)$ is not a discrete space, then 
$\theta_t$ is not uniformly continuous for $t \in (1, \infty)$,  
since $\frac{\lambda_\kappa(2ts)}{\lambda_\kappa(2s)} \to \infty$ $(s \to \infty)$ in Lemma~\ref{lem_theta_t}\,(1)(i). 
\end{remark}

\begin{lemma}\label{lem_Alexander_trick} Let $h, k \in {\cal H}(C_\kappa(X, d))$.  
\begin{enumerate}
\item 
\begin{itemize}
\item[(i)\,] $h_1= h$,  \ \ $\id_t = \id$, \ \ $(kh)_t = k_th_t$, \ \ $(h^{-1})_t = (h_t)^{-1}$ \ \ and \ \ $h_{st} = (h_s)_t$ \ \ $(s, t \in [0, \infty))$. 
\vskip 1mm 
\item[(ii)] For any $t \in [0, 1]$ \\   
{\rm (a)} $\lambda_\kappa(\tilde{d}_\kappa (h_t, k_t)) \leq t \lambda_\kappa(\tilde{d}_\kappa (h, k))$, \ \ 
$\lambda_\kappa(\tilde{d}_\kappa (h_t, \id)) \leq t \lambda_\kappa(\tilde{d}_\kappa (h, \id))$, \\ 
{\rm (b)} $\tilde{d}_\kappa(h_t, k_t) \leq \tilde{d}_\kappa(h, k)$, \ \ 
$\tilde{d}_\kappa(h_t, \id) \leq \tilde{d}_\kappa(h, \id)$.  
\end{itemize}

\item 
\begin{itemize}
\item[(i)\,] $\Phi(\id) = \id$, \ \ $\Phi(kh) = \Phi(k)\Phi(h)$ \ \ and \ \ $\Phi(h^{-1}) = \Phi(h)^{-1}$. 
\vskip 1mm 
\item[(ii)] $\tilde{\tilde{d}}_\kappa(\Phi(h), \Phi(k)) = \tilde{d}_\kappa(h, k)$ \\
$($where $\tilde{\tilde{d}}_\kappa$ is the sup-metric induced from the metric $\tilde{\tilde{d}}_\kappa$ on $C_\kappa(X, d) \times [0,1])$. 
\end{itemize}

\item If $h \in {\cal H}_b(C_\kappa(X, d))$, then 
\item[]
\begin{itemize}
\item[(i)\ ] $h_t \in {\cal H}_b(C_\kappa(X, d))$ $(t \in [0, 1])$, \ {\rm (ii)} $\eta_h$ is continuous at $t = 0$ \ \ and  
\item[(iii)] $\Phi(h)$ is a level-wise bounded isotopy on $C_\kappa(X, d)$ from $\id$ to $h$. 
\end{itemize}

\item $\Phi(h)$ is a uniform isotopy on $C_\kappa(X, d)$ if and only if 
$\eta_h|_{[0,1]}$ is continuous  and $h_t \in {\cal H}^u(C_\kappa(X, d))$ $(t \in [0, 1])$. 
In this case, $\Phi(h)$ and $h_t$ $(t \in [0, 1])$ are bounded.  
\end{enumerate}
\end{lemma}

\begin{proof} (1)(ii)(a) 
From Lemma~\ref{lem_theta_t}\,(2)(ii) it follows that 
for any $t \in (0, 1]$ and $u \in C_\kappa(X, d)$ \\[1mm] 
\hspace*{15mm} 
$\lambda_\kappa(\tilde{d}_\kappa (h_t(u), k_t(u))) 
= \lambda_\kappa(\tilde{d}_\kappa (th(\frac{1}{t} u), t k(\frac{1}{t} u))) 
\leq t \lambda_\kappa(\tilde{d}_\kappa (h(\frac{1}{t} u), k(\frac{1}{t} u))) 
\leq t \lambda_\kappa(\tilde{d}_\kappa (h, k)).$

(2)(ii) The claim follows from the following observations; \ \ For any $u \in C_\kappa(X, d)$ and $t \in [0,1]$ \\[2mm] 
\hspace*{8mm} 
$\begin{array}[c]{l}
\tilde{\tilde{d}}_\kappa(\Phi(h)(u,t), \Phi(k)(u,t)) 
= \tilde{\tilde{d}}_\kappa((h_t(u),t), (k_t(u),t)) 
= \tilde{d}_\kappa(h_t(u), k_t(u)) \leq \tilde{d}_\kappa(h_t, k_t)\leq \tilde{d}_\kappa(h, k), \\[2mm] 
\tilde{d}_\kappa(h(u), k(u)) = \tilde{\tilde{d}}_\kappa(\Phi(h)(u,1), \Phi(k)(u,1)) \leq \tilde{\tilde{d}}_\kappa(\Phi(h), \Phi(k)).
\end{array}$ 
\vskip 2mm 
(3) The claims (i) and (ii) follow from (1)(ii) (b) and (a) respectively. 

(iii) 
The function $\Phi(h)$ is continuous at  any $(u,0)$ $(u \in C_\kappa(X, d))$ by (ii) and 
at any $(u,t) \in C_\kappa(X, d) \times (0,1]$ by the definition of $\Phi(h)$ itself 
and the continuity of the scalor multiplication $C_\kappa(X, d) \times [0,\infty) \ni (v, s) \mapsto sv \in C_\kappa(X, d)$. 

(4) The assertion follows from Lemma~\ref{lem_A_unform_isotopy}. 
Note that the open subgroup ${\cal H}^u_b(C_\kappa(X, d))$ includes 
any continuous path in ${\cal H}^u(C_\kappa(X, d))$ emanating from $\id$. 
\end{proof}

For each $h \in {\cal H}_b(C_\kappa(X, d))$ 
the isotopy $\Phi(h) = \{ h_t \}_{t \in [0,1]}$ on $C_\kappa(X, d)$ from $\id$ to $h$ is called the Alexander isotopy for $h$. 
By Lemma~\ref{lem_Alexander_trick}\,(2)(ii) we obtain an isometric embedding  
$$\Phi : {\cal H}_b(C_\kappa(X, d)) \ni h \longmapsto \Phi(h) \in {\cal H}^{lb}(C_\kappa(X, d) \times [0,1], \tilde{\tilde{d}}_\kappa)^I,$$ 
which constitutes a contraction* of ${\cal H}_b(C_\kappa(X, d))$. 
However, in Example~\ref{exp_Alexander_isotopy_kappa-cones} below, 
for any compact Lipschitz metric manifold $(X,d)$ we construct an example of $k \in {\cal H}^u_b(C_\kappa(X, d))$ 
for which $\eta_k|_{[0,1]}$ is not continuous. 
This means that the function associated to the contraction* $\Phi$, 
$$\mbox{$\phi : {\cal H}_b(C_\kappa(X, d)) \times [0,1] \lra {\cal H}_b(C_\kappa(X, d))$: \ \ $\phi(h, t) = h_t$}$$
is not necessarily continuous on ${\cal H}^u_b(C_\kappa(X, d)) \times [0,1]$. 

In the case that $\kappa = 0$ Alexander isotopies have some specific properties: 
In this case, $\theta_t$ is a $t$-similar transformation on $C_0(X, d)$ for each $t> 0$. 
Consider the function 
$$\xi_t : {\cal H}(C_0(X, d)) \lra {\cal H}(C_0(X, d)) : \ \ \xi_t(h) = h_t \hspace{5mm} (t \in [0, \infty))$$ 

\begin{lemma}\label{lem_Alexander_trick_k=0} Let $h, k \in {\cal H}(C_0(X, d))$ and $t \in [0, \infty)$.
\begin{enumerate}

\item $\tilde{d}_0(h_t, k_t) = t\tilde{d}_0(h,k)$, \ \ $\tilde{d}_0(h_t, \id) = t\tilde{d}_0(h, \id)$. 
\item The map $\xi_t$ is a $t$-similar embedding for $t > 0$. 
\item If $h \in {\cal H}_b(C_0(X, d))$, then $h_t \in {\cal H}_b(C_0(X, d))$. 
If $h \in {\cal H}^u(C_0(X, d))$, then $h_t \in {\cal H}^u(C_0(X, d))$. \\[0.5mm]
If $h \in {\cal H}^u_b(C_0(X, d))$, then \ $\Phi(h) \in {\cal H}^{lbu}\big(C_0(X, d) \times [0,1], \tilde{\tilde{d}}_0 \big)^I.$ 

\item 
\begin{itemize}
\item[(i)\ ] $\eta_h$ is continuous at $t = 0$ if and only if $h \in {\cal H}_b(C_0(X, d))$. 
\item[(ii)\,] If $\eta_h$ is continuous at a point in $(0, \infty)$, then $\eta_h$ is continuous at any point in $(0, \infty)$. 
\item[(iii)] $\Phi(h)$ is a uniform isotopy if and only if $h \in {\cal H}^u(C_0(X, d))$ and $\eta_h$ is continuous. \hfill $\square$  
\end{itemize}
\end{enumerate}
\end{lemma}

The contraction* $\Phi$ of ${\cal H}_b(C_0(X, d))$ restrcts to a contraction* of ${\cal H}^u_b(C_0(X, d))$, 
$$\Phi : {\cal H}^u_b(C_0(X, d)) \to {\cal H}^{lbu}\big(C_0(X, d) \times [0,1], \tilde{\tilde{d}}_0 \big)^I.$$ 
However, the function associated to this contraction* $\Phi$, 
$$\mbox{$\phi : {\cal H}^u_b(C_0(X, d)) \times [0,1] \to {\cal H}^u_b(C_0(X, d))$: \ \ $\phi(h, t) = h_t$}$$
is not continuous for any compact Lipschitz metric manifold $(X,d)$. 


\subsection{Contraction of subgroups} \mbox{} 

One can expect that the function $\phi$ restricts to a contraction of a suitable subgroup of ${\cal H}^u_b(C_\kappa(X, d))$.  
In this subsection we discuss this question. 
Consider the following subset of ${\cal H}^u_b(C_\kappa(X, d))$: 
$${\cal G} = {\cal G}(C_\kappa(X, d)) := \{ h \in {\cal H}(C_\kappa(X, d)) \mid \mbox{$\Phi(h)$ is a uniform isotopy.} \}.$$  
Recall that (i) $\Phi(h)$ is a uniform isotopy if and only if $\eta_h|_{[0,1]}$ is continuous and $h_t \in {\cal H}^u(C_\kappa(X, d))$ $(t \in [0, 1])$
and (ii) if $\Phi(h)$ is a uniform isotopy, then $h_t \in {\cal H}_b(C_\kappa(X, d))$ $(t \in [0,1])$. 

\begin{proposition}\label{prop_G} {\rm (1)} ${\cal G}$ is a subgroup of ${\cal H}^u_b(C_\kappa(X, d))$. 
{\rm (2)} If $h \in {\cal G}$, then $h_s \in {\cal G}$ for any $s \in [0, 1]$. 
\begin{enumerate}
\item[(3)] The function $\phi$ restricts to a contraction $\phi : {\cal G} \times [0,1] \to {\cal G}$. 
\end{enumerate}
\end{proposition}

\begin{proof}
(1) The claim follows from Lemma~\ref{lem_Alexander_trick}\,(2)(i) and (4). 

(2) By Lemma~\ref{lem_Alexander_trick}\,(1)(i) we have $\eta_{h_s}(t) = (h_s)_t = h_{st}$ $(t \in [0,1])$.

(3) By the definition of ${\cal G}$ the map $\Phi$ restricts to the map $\Phi : {\cal G} \to {\cal H}^u_b(C_\kappa(X, d) \times [0,1], \tilde{\tilde{d}}_\kappa)^I$. 
Proposition~\ref{prop_exp_low} implies that the associated function $\phi : {\cal G} \times [0,1] \to {\cal H}^u_b(C_\kappa(X, d))$ is continuous.  
(We can also consider the map $\psi : {\cal G} \to {\cal C}([0,1], {\cal H}^u_b(C_\kappa(X, d)))$ : $\psi(h) = \eta_h|_{[0,1]}$,  
which is shown to be an isometric embedding with respect to the sup-metrics.)
By (2) we have $\phi( {\cal G} \times [0,1] ) \subset  {\cal G}$. 
\end{proof} 

\begin{remark}\label{rmk_max}
(1) The subset ${\cal G}$ is maximal in ${\cal H}^u_b(C_\kappa(X, d))$ in the following sense: 
\begin{enumerate}
\item[$(\natural)$] If ${\cal S} \subset {\cal H}(C_\kappa(X, d))$, 
$\phi : {\cal S} \times [0,1] \to {\cal H}(C_\kappa(X, d))$ is continuous
and $\phi({\cal S} \times [0,1]) \subset {\cal H}^u(C_\kappa(X, d))$, 
then ${\cal S} \subset {\cal G}$. 

\item[(2)] Suppose ${\cal S} \subset {\cal G}$ and ${\cal H}$ is the subgroup of ${\cal G}$ generated by ${\cal S}$. 
If $\phi({\cal S} \times [0,1]) \subset {\cal H}$, then $\phi({\cal H} \times [0,1]) \subset {\cal H}$ and 
the contraction $\phi$ of ${\cal G}$ restricts to a contraction of ${\cal H}$. 
\end{enumerate}
\end{remark}

\begin{proof}
(1) The assertion follows from the definition of ${\cal G}$ itself. 

(2) Any $h \in {\cal H}$ can be represented as 
$$\mbox{$h = g_1^{\ \epsilon_1} \cdots g_m^{\ \epsilon_m}$ \hspace{5mm} 
for some \ $g_1, \cdots, g_m \in {\cal S}$ \ and \ $\epsilon_1, \cdots, \epsilon_m = \pm 1$.}$$ 
Then, for any $t \in [0,1]$ we have \ $(g_1)_t, \cdots, (g_m)_t \in {\cal H}$ \ and \ 
$h_t =  (g_1)_t^{\ \epsilon_1} \cdots (g_m)_t^{\ \epsilon_m} \in {\cal H}$. 
\end{proof} 

Next we show that the group ${\cal G}$ includes the normal subgroup ${\cal H}_0(C_\kappa(X, d))$ of ${\cal H}^u_b(C_\kappa(X, d))$ (cf. Section 2.4).

\begin{proposition}\label{prop_contraction_H_0} 
Both ${\cal H}_0(C_\kappa(X, d))$ and ${\cal H}_c(C_\kappa(X, d))$ are normal subgroups of ${\cal G}$ and 
the contraction $\phi$ of ${\cal G}$ restricts to the contractions 
$$\mbox{$\phi : {\cal H}_0(C_\kappa(X, d)) \times [0,1] \to {\cal H}_0(C_\kappa(X, d))$\ \ \text{and} \ \ 
$\phi : {\cal H}_c(C_\kappa(X, d)) \times [0,1] \to {\cal H}_c(C_\kappa(X, d))$.}$$
\end{proposition}

This proposition follows from the next lemma. 

\begin{lemma}\label{lem_G_phi} {\rm (1)}  If $h \in {\cal H}_c(C_\kappa(X, d))$, then $h_t \in {\cal H}_c(C_\kappa(X, d))$ $(t \in [0,\infty))$. 
\begin{enumerate}
\item[(2)] If $h \in {\cal H}_0(C_\kappa(X, d))$, then 
{\rm (i)} $h_t \in {\cal H}_0(C_\kappa(X, d))$ $(t \in [0,1])$ and \ {\rm (ii)} $\eta_h|_{[0,1]}$ is continuous. 
\end{enumerate}
\end{lemma} 

\begin{proof} 
(1) If $h \in {\cal H}_c(C_\kappa(X, d))$, then 
$h_t = \theta_t h \theta_t^{-1} \in {\cal H}_c(C_\kappa(X, d))$ for any $t \in (0, \infty)$. 

(2) First we show that 
\begin{itemize}
\item[$(\dagger)$] for any $\e > 0$ there exists $r > 0$ such that $\tilde{d}_\kappa(h_t(u), u) < \e$ 
for any $t \in [0,1]$ and $u \in C_\kappa(X,d)_r$. 
\end{itemize}
In fact, since $h \in {\cal H}_0(C_\kappa(X,d))$, for any $\e > 0$ 
there exists a compact set $K \subset C_\kappa(X,d)$ such that 
$\tilde{d}_\kappa(h(u), u) < \e$ $(u \in C_\kappa(X,d) - K)$. 
There exists $r > 0$ such that $K \subset O_r(\bs{0})$.  
For any $t \in (0,1]$ and $u \in C_\kappa(X,d)_r$, since $\frac{1}{t}u \in C_\kappa(X,d)_r \subset C_\kappa(X,d) - K$, it follows that 
$$\mbox{$\tilde{d}_\kappa(h_t(u), u) 
= \tilde{d}_\kappa(th(\frac{1}{t}u), t(\frac{1}{t}u))
\leq \tilde{d}_\kappa(h(\frac{1}{t}u), \frac{1}{t}u)) < \e.$}$$ 

(i) For any $r \geq 0$ the closed $r$-neighborhood $B_r(\bs{0})$ is compact, 
since $B_r(\bs{0}) = \{ tx \in C_\kappa(X,d) \mid x \in X, t \in [0,r] \}$ and $X$ is assumed to be compact. 
Thus, the claim (i) follows from $(\dagger)$. 

(ii) We show that $\eta_h|_{[0,1]}$ is uniformly continuous. Take any $\e > 0$. 

(a) There exists $r > 0$ such that $\tilde{d}_\kappa(h_t(u), u) < \e/2$ $(t \in [0,1], u \in C_\kappa(X,d)_r)$. 
Then, for any $s, t \in [0,1]$ it follows that 
$$\tilde{d}_\kappa(h_s(u), h_t(u)) \leq \tilde{d}_\kappa(h_s(u), u) + \tilde{d}_\kappa(u, h_t(u)) < \e \ \ 
(u \in C_\kappa(X,d)_r).$$ 

(b) Since $h \in {\cal H}_0(C_\kappa(X, d)) \subset {\cal H}_b(C_\kappa(X, d))$, it follows that 
$\Phi(h)$ is an isotopy and the map $C_\kappa(X, d) \times [0,1] \ni (u,t) \mapsto h_t(u) \in C_\kappa(X, d)$ is continuous.
Since $B_r(\bs{0})$ is compact, there exists $\delta > 0$ such that if $s, t \in [0,1]$ and $|s-t| < \delta$ then 
$\tilde{d}_\kappa(h_s(u), h_t(u)) < \e$ $(u \in B_r(\bs{0}))$. 

These observations mean that  
$\tilde{d}_\kappa(h_s, h_t) \leq \e$ if $s,t \in [0,1]$ and $|s-t| < \delta$. 
\end{proof}


\subsection{Examples} \mbox{}

First we focus on the Euclidean space $({\Bbb R}^n,d)$. 
Since it has the canonical $0$-cone structure ${\Bbb R}^n = C_0({\Bbb S}^{n-1})$, 
the results in the previous sections are applied to the group ${\cal H}^u_b({\Bbb R}^n)$. 
The next example shows that the function $\eta_h|_{[0,1]}$ is not necessarily continuous for $h \in {\cal H}^u_b(\IR^n)$. 

\begin{example}\label{exp_Alexander_isotopy_IR^n} (the $\IR^n$-case) 
For any non-zero vector $v \in \IR^n$ we can find $h \in {\cal H}^u_b(\IR^n)$ such that 
\begin{itemize}
\item[] 
\begin{itemize}
\item[(i)\,] $h((2k+1)v) = (2k+1)v$ \ $(k \in \IN)$ \ \ and 
\item[(ii)] there exists $c > 0$ such that \ $d(h(2kv), 2kv) > c$ \ $(k \in \IN)$. 
\end{itemize}
\end{itemize}
For any such $h$, the function $\eta_h|_{[0,1]}$ is not continuous at $t = 1$.
\end{example}

\begin{proof}
Assume the contrary, that $\eta_h|_{[0,1]}$ is continuous at $t = 1$. 
Then, there exists $t_0 \in [0, 1)$ such that $d(h_t, h) < c$ for any $t \in (t_0, 1]$. 
Since $\frac{2k}{2k+1} \to 1$ $(k \to \infty)$, there exists $k \in \IN$ such that $t := \frac{2k}{2k+1} > t_0$. \\[1mm] 
It follows that \ \ $h_t(2kv) = th\left(\mbox{$\frac{1}{t}2kv$}\right) = th((2k+1)v) = t(2k+1)v = 2kv$ \ \ so that 
$$d(h_t, h) \geq d(h_t(2kv), h(2kv)) = d(2kv, h(2kv)) > c.$$ 
This contradicts $d(h_t, h) < c$. 
\end{proof} 

In \cite{Ya} we showed that ${\cal H}^u_b({\Bbb R}^n)$ is contractible in the usual sense. 
There we adopted the following strategy: 
(1) Since $\IR^n$ admits a metric covering projection onto the flat torus, $\IR^n$ has the property (LD) and 
there exists a contraction $\chi_t$ of a small open ball ${\cal H}^u(\id, \e; \IR^n)$ in ${\cal H}^u_b({\Bbb R}^n)$. 
(2) Given $\alpha > 0$, using $\chi_t$ and a similarity transformation $\theta_\gamma$, 
we can find $\beta > 0$ and 
a contraction $\psi_t$ of the open ball ${\cal H}^u(\id, \alpha; \IR^n)$ in ${\cal H}^u(\id, \beta; \IR^n)$. 
(3) The iteration of (2) yields 
a sequence of contractions $\psi^i_t$ of ${\cal H}^u(\id, \alpha_i; \IR^n)$ in ${\cal H}^u(\id, \alpha_{i+1}; \IR^n)$ 
for some increasing sequence $\alpha_i \in \IR$ $(i \in \IN)$. 
A contraction of ${\cal H}^u_b({\Bbb R}^n)$ is obtained by composing these contractions $\psi^i_t$. 
In this argument, we need not change the scale factor $\gamma$ of the similarity transformation $\theta_\gamma$ continuously.  

By Proposition~\ref{prop_G} the Alexander isotopies induce a contraction of the subgroup ${\cal G}(\IR^n)$ of ${\cal H}^u_b({\Bbb R}^n)$,  
$$\phi : {\cal G}(\IR^n) \times [0,1] \to {\cal G}(\IR^n) : \phi(h,t) = h_t.$$
This group includes the translations $h_v$ $(v \in \IR^n)$ defined by $h_v(x) = x + v$ $(x \in \IR^n)$.  

\begin{example}\label{exp_IR^n_A} The group ${\cal G}(\IR^n)$ includes the following subgroups, 
each of which is preserved by the contraction $\phi$ of ${\cal G}(\IR^n)$:  
\begin{tabular}[t]{l}
(a) ${\cal A}(\IR^n) = \{ h_v \mid v \in \IR^n \}$ \ \ \ (b) ${\cal H}_c(\IR^n)$ \ \ \ (c) ${\cal H}_0(\IR^n)$ \\[1.5mm] 
(d) ${\cal H}_c = \{ gh_v \mid g \in {\cal H}_c(\IR^n), v \in \IR^n \}$ \ \  
(e) ${\cal H}_0 = \{ gh_v \mid g \in {\cal H_0}(\IR^n), v \in \IR^n \}$. 
\end{tabular} 
\end{example}

\begin{proof} 
(a) For each $v \in \IR^n$ it follows that 
$h = h_v \in {\cal H}^u_b({\Bbb R}^n)$, $h_t = h_{tv}$, $d(h_t, h_s) = \|v\||t - s|$ $(t, s \in [0,\infty))$ and $\eta_h$ is continuous. 
These observation imply the claim on ${\cal A}(\IR^n)$. 

(b)(c) The claims follow from Proposition~\ref{prop_contraction_H_0}. 

(d)(e) Since ${\cal H}_c(\IR^n)$ and ${\cal H}_0(\IR^n)$ are normal subgroups of ${\cal H}^u_b({\Bbb R}^n)$, 
it is seen that ${\cal H}_c \subset {\cal H}_0$ coincide with the subgroups of  ${\cal H}^u_b({\Bbb R}^n)$ generated by 
the unions ${\cal H}_c(\IR^n) \cup {\cal A}(\IR^n)$ and ${\cal H}_0(\IR^n) \cup {\cal A}(\IR^n)$ respectively. 
The claims now follow from (a), (b), (c) and Remark~\ref{rmk_max}\,(2) 
\end{proof} 

The author was informed that R.~A.~McCoy had shown that the Alexander isotopies induce a contraction of ${\cal H}_c(\IR^n)$. 

\begin{remark}\label{rmk_compact-open_top} 
Let ${\cal H}_c({\Bbb R}^n)_{co}$ denote the group ${\cal H}_c({\Bbb R}^n)$ endowed with the compact-open topology. 
For $n \geq 2$, the function $\phi : {\cal H}_c({\Bbb R}^n)_{co} \times [0,1] \to {\cal H}_c({\Bbb R}^n)_{co}$ is not continuous at $(\id_{{\Bbb R}^n}, 0)$. 
\end{remark}

\begin{proof} 
Recall that a fundamental neighborhood system of $\id$ in ${\cal H}_c({\Bbb R}^n)_{co}$ is given by 
$$\mbox{${\cal U}(\id, K, \e) = \{ h \in {\cal H}_c({\Bbb R}^n) \mid d(h(x), x) < \e \ (x \in K)\}$
\hspace{3mm} $(K \in {\cal K}(\IR^n), \e > 0)$}.$$

Contrary, suppose $\phi$ is continuous at $(\id_{{\Bbb R}^n}, 0)$. 
Consider the points $e_1 = (1, 0, \cdots, 0)$ and $e_2 = (0, 1, \cdots, 0)$ in $\IR^n$. 
Since $\phi(\id, 0) = \id$, 
there exist $r > 0$ and $\delta > 0$ such that 
$\phi(\U(\id, B_r(0), \delta) \times [0, \delta)) \subset \U(\id, \{ e_1 \}, 1)$. 
Choose $s > r$ with $\frac{1}{s} < \delta$ and 
find $h \in {\cal H}_c({\Bbb R}^n) $ such that $h(s e_1)= s  e_2$ and 
$h = \id$ on $B_r(0) \cup ({\Bbb R}^n - B_{s+1}(0))$. 
Then, $h \in \U(\id, B_r(0), \delta)$ and $t := \frac{1}{s} \in [0,\delta)$, so 
$h_t \in \U(\id, \{ e_1 \}, 1)$ and $e_2 = h_t(e_1) \in O_1(e_1)$, which is a contradiction. 
\end{proof} 

The argument in Example~\ref{exp_Alexander_isotopy_IR^n} extends to the $\kappa$-cone case. 

\begin{example}\label{exp_Alexander_isotopy_kappa-cones} (the $\kappa$-cone case) 
Consider the following condition on a metric space $(Y, \rho)$: 
\begin{itemize}
\item[$(\ast)$] 
There exists a point $z \in Y$, $L \geq 1$ and $\e_0 > 0$ such that 
for any $\e \in (0, \e_0)$ there exists an isotopy $\{ f_s \}_{s \in [0,1]}$ on $Y$ 
which satisfies the following condition : \\
\hspace*{-3mm} $(\flat)(z, L, \e)$ : 
\begin{tabular}[t]{@{}l}
for each $s, t \in [0,1]$ \  
(i) $f_s$ is a $L$-Lipschitz homeomorphism of $(Y,\rho)$, (ii) $f_1 = \id_Y$, \\[2mm]
\hspace{-2mm} (iii) ${\rm supp}\, f_s \subset O_\e(z)$, \  
(iv) $\rho(f_s^{\, \delta}, f_t^{\, \delta}) \leq \e L|s-t|$ $(\delta = \pm 1)$ and (v) $\rho(f_0(z),z) \geq \frac{\e}{L}$.
\end{tabular}  
\end{itemize}
\vskip 1mm 

For example, the Euclidean space $\IR^n$ satisfies the condition $(\ast)$. 
Indeed, we can find an isotopy $h_s$ on $\IR^n$ which satisfies the condition $(\flat)(\bs{0}, L, 1)$ for some $L \geq 1$. 
Then, for any $\e > 0$ the isotopy $f_s = \theta_\e h_s \theta_{1/\e}$ $(s \in [0,1])$ satisfies the condition $(\flat)(\bs{0}, L, \e)$. 
Note that (a) the condition $(\ast)$ is preserved under any bi-Lipschitz equivalence and 
(b) the condition $(\ast)$ is a local property, 
that is, when $U$ is a neighborhood of a point $z$ in $(Y, \rho)$, 
the metric space $(Y, \rho)$ satisfies the condition $(\ast)$ with respect to the point $z$ if and only if 
so does $(U, \rho|_U)$.  
In fact, 
if $\phi : (Y,\rho) \cong (Z, \sigma)$ is a $K$-Lipschitz homeomorphism and 
$f_s$ is an isotopy on $Y$ which satisfies the condition  $(\flat)(y, L, \e/K)$, then 
$g_s = \phi f_s \phi^{-1}$ is an isotopy on $Z$ which satisfies the condition  $(\flat)(\phi(y), K^2 L, \e)$. 
Thus, any Lipschitz metric manifold also satisfies the condition $(\ast)$. 

If a compact metric space $(X, d)$ satisfies the condition $(\ast)$, 
then we can construct 
$h \in {\cal H}^u_b(C_\kappa(X, d))$ for which 
the function $\eta_h|_{[0,1]}$ is not continuous at $1$. 
\end{example} 

\begin{proof} 
Let $c_\kappa = \frac{1}{2}$ if $\kappa = 0$ and $c_\kappa = \frac{\sqrt{-\kappa}}{2}$ if $\kappa < 0$. 
Suppose $(X, d)$ satisfies the condition $(\ast)$ with respect to a point $z \in X$, $L \geq 1$ and $\e_0 > 0$. 
We may assume that $L> \frac{1}{c_\kappa}$, $L > \e_0$ and $L\lambda_\kappa(2) \e_0 > 1$. 

(1) For each $k \in \IN$ define $\e_k \in (0, \e_0)$ by $L\lambda_\kappa(4k+2) \e_k = 1$ and 
take  an isotopy $\{ f^k_s \}_{s \in [0,1]}$ on $X$ which satisfies the condition $(\flat)(z, L,\e_k)$. 
Define $h \in {\cal H}(C_\kappa(X, d))$ by 
$$h(rx) = 
\left\{ \hspace{-1mm} 
\begin{array}[c]{ll}
rf_s^k(x) & (x \in X, \ r = 2k \pm s \ (s \in [0,1], \ k \in \IN)), \\[2mm] 
rx & (x \in X, \ r \in [0,1]).
\end{array}\right.$$
We have to show that $h \in {\cal H}^u_b(C_\kappa(X, d))$. 

(i)  For $\delta = \pm1$, \ $x, y \in X$ \ and \ $r = 2k + \epsilon s$, \ $q = 2k + \epsilon t$ \ $(k \in \IN, \ s,t \in [0,1], \ \epsilon = \pm1)$   
\begin{itemize}
\item[(a)] $\tilde{d}_\kappa(h^\delta(rx), h^\delta(ry)) = \tilde{d}_\kappa(r(f_s^k)^\delta(x), r(f_s^k)^\delta(y)) \leq L\tilde{d}_\kappa(rx, ry)$,
\item[(b)] $\tilde{d}_\kappa(r(f_s^k)^\delta(x), r(f_t^k)^\delta(x)) \leq L\tilde{d}_\kappa(rx, qx) \leq L$, \ \ \  
$\tilde{d}_\kappa(h(rx), rx) = \tilde{d}_\kappa(rf_s^k(x), rx) \leq L$, 
\item[(c)] $\tilde{d}_\kappa(h^\delta(rx), h^\delta(qx)) \leq (L+1) \tilde{d}_\kappa(rx, qx)$.
\end{itemize}
These inequalities are deduced from the following observations: 

\begin{itemize}
\item[(a$'$)] Since $(f_s^k)^\delta : (X, d_\pi) \to (X, d_\pi)$ is also $L$-Lipschitz, we have
$d_\pi((f_s^k)^\delta(x), (f_s^k)^\delta(y)) \leq L d_\pi(x, y)$ \\
and so, from Remark~\ref{rem_kappa-cone}\,(3) and (1)(ii) it follows that \\
$\begin{array}[t]{@{}l@{\ }l}
\lambda_\kappa \big(\tilde{d}_\kappa(r(f_s^k)^\delta(x), r(f_s^k)^\delta(y))\big)
& = \lambda_\kappa(2r) \sin \frac{1}{2} d_\pi((f_s^k)^\delta(x), (f_s^k)^\delta(y)) \\[1.5mm] 
& \leq \lambda_\kappa(2r) L \sin \frac{1}{2} d_\pi(x, y)
= L \lambda_\kappa \big(\tilde{d}_\kappa(rx, ry)\big) \leq \lambda_\kappa \big(L\tilde{d}_\kappa(rx, ry)\big). 
\end{array}$  
\vskip 2mm 

\item[(b$'$)] Since $d((f_s^k)^\delta, (f_t^k)^\delta) \leq \e_k L |s-t|$, we have \\[1mm] 
\hspace*{5mm} $\sin \frac{1}{2} d_\pi((f_s^k)^\delta(x), (f_t^k)^\delta(x)) 
\leq d_\pi((f_s^k)^\delta(x), (f_t^k)^\delta(x)) \leq d((f_s^k)^\delta(x), (f_t^k)^\delta(x)) \leq \e_k L|s-t|$. \\[1mm] 
Since $c_\kappa u \leq \lambda_\kappa(u)$ $(u \geq 0)$, $L\lambda_\kappa(4k+2) \e_k = 1$ and $L \geq \frac{1}{c_\kappa}$, 
it follows that \\[2mm] 
\hspace*{5mm} $\begin{array}[t]{@{}l}
c_\kappa \tilde{d}_\kappa(r(f_s^k)^\delta(x), r(f_t^k)^\delta(x)) \leq
\lambda_\kappa \big(\tilde{d}_\kappa(r(f_s^k)^\delta(x), r(f_t^k)^\delta(x))\big) \\[2mm] 
\hspace{15mm} = \lambda_\kappa(2r) \sin \frac{1}{2} d_\pi((f_s^k)^\delta(x), (f_t^k)^\delta(x)) \leq \lambda_\kappa(4k+2) \e_k L|s-t| = |s-t|. \\[2mm] 
\therefore \ \tilde{d}_\kappa(r(f_s^k)^\delta(x), r(f_t^k)^\delta(x)) \leq L|s-t| = L|r-q| = L\tilde{d}_\kappa(rx, qx) \leq L. 
\end{array}$ \\[2mm] 
Letting $t=1$, we have the 2nd inequality. 

\vskip 2mm 

\item[(c$'$)] Since $\tilde{d}_\kappa(r(f_t^k)^\delta(x), q(f_t^k)^\delta(x)) = |r-q| = \tilde{d}_\kappa(rx, qx)$, 
from (i)(b) it follows that \\[2mm] 
$\begin{array}[c]{l}
\tilde{d}_\kappa(h^\delta(rx), h^\delta(qx)) = \tilde{d}_\kappa(r(f_s^k)^\delta(x), q(f_t^k)^\delta(x)) \\[2mm] 
\hspace{10mm} \leq  
\tilde{d}_\kappa(r(f_s^k)^\delta(x), r(f_t^k)^\delta(x))
+ \tilde{d}_\kappa(r(f_t^k)^\delta(x), q(f_t^k)^\delta(x)) \\[2mm] 
\hspace{10mm} \leq L \tilde{d}_\kappa(rx, qx) + \tilde{d}_\kappa(rx, qx) = (L+1) \tilde{d}_\kappa(rx, qx).
\end{array}$
\end{itemize}

(ii) The 2nd inequality in (i)(b) implies that $\tilde{d}_\kappa(h, \id) \leq L$.

(iii) We show that $h$ is a $(2L+1)$-Lipschitz homeomorphism. 
To see that $h^\delta$ $(\delta = \pm1)$ are $(2L+1)$-Lipschitz, 
take any $rx, qy \in C_\kappa(X, d)$ $(x,y \in X, r, q \in [0,\infty), r \leq q)$. 
From (i)(c) it follows that 
\begin{itemize}
\item[(d)] $\tilde{d}_\kappa(h^\delta(ry), h^\delta(qy)) \leq (L+1)\tilde{d}_\kappa(ry, qy)$. 
\end{itemize}
Indeed, if $k-1 \leq r \leq k \leq l \leq q \leq l+1$ $(k, l \in \IN)$, then 
\vskip 2mm 
\hspace*{5mm} 
$\begin{array}[c]{l@{\ }l} 
\tilde{d}_\kappa(h^\delta(ry), h^\delta(qy)) 
&\leq \tilde{d}_\kappa(h^\delta(ry), h^\delta(ky)) 
+ \sum_{i=k}^{l-1} \tilde{d}_\kappa(h^\delta(iy), h^\delta((i+1)y))
+ \tilde{d}_\kappa(h^\delta(ly), h^\delta(qy)) \\[2mm]  
&\leq  (L+1)\tilde{d}_\kappa(ry, ky) 
+ \sum_{i=k}^{l-1} (L+1)\tilde{d}_\kappa(iy, (i+1)y) 
+ (L+1)\tilde{d}_\kappa(ly, qy) \\[2mm]  
& = (L+1)\big((k-r) + \sum_{i=k}^{l-1} 1 + (q-l)\big) = (L+1)(q - r) 
= (L+1)\tilde{d}_\kappa(ry, qy) 
\end{array}$ \\[1mm] 
Then, by (a) and (d) we have 
\vskip 2mm 
\hspace*{5mm} 
$\begin{array}[t]{l@{\ }l} 
\tilde{d}_\kappa(h^\delta(rx), h^\delta(qy)) 
& \leq \tilde{d}_\kappa(h^\delta(rx), h^\delta(ry)) + \tilde{d}_\kappa(h^\delta(ry), h^\delta(qy)) \\[2mm] 
& \leq L \tilde{d}_\kappa(rx, ry) + (L+1)\tilde{d}_\kappa(ry, qy) = (2L+1)\tilde{d}_\kappa(rx, qy).
\end{array}$
\vskip 2mm 

(2) Next we show that the function $\eta_h|_{[0,1]}$ is not continuous at $1$. 

(i) The map $h$ has the following properties: 
\begin{itemize}
\item[(a)] $h((2k+1)z) = (2k+1)z$ for each $k \in \IN$.
\item[(b)] There exists $a > 0$ such that $\tilde{d}_\kappa(h(2kz), 2kz) \geq a$ for each $k \in \IN$. 
\end{itemize}

The claim (b) is verified as follows: 
Let $\ds \gamma := \inf_{k \in \IN} \mbox{$\frac{\lambda_\kappa(4k)}{\lambda_\kappa(4k+2)}$}$. 
Then $\gamma > 0$ since 
$\ds \lim_{k \to \infty} \mbox{$\frac{\lambda_\kappa(4k)}{\lambda_\kappa(4k+2)}$} =1 \text{ if $\kappa = 0$}$ and 
$= e^{-\sqrt{-\kappa}} \text{ if $\kappa < 0$}$. 
Since $L\lambda_\kappa(4k+2) \e_k = 1$ and $d_\pi(f_0^k(z), z) \geq \frac{\e_k}{L}$ by $(\flat)(z, L,\e_k)$ (v), 
it follows that 
$$\begin{array}[t]{l@{\ }l}
\lambda_\kappa \big(\tilde{d}_\kappa(h(2kz), 2kz)\big) 
& = \ \lambda_\kappa \big(\tilde{d}_\kappa(2kf_0^k(z), 2kz)\big) 
\ = \ \lambda_\kappa(4k) \sin \frac{1}{2} d_\pi(f_0^k(z), z) \\[2.5mm] 
& \geq \ \lambda_\kappa(4k) \sin \frac{1}{2} \frac{\e_k}{L} \ 
\geq \ \gamma \lambda_\kappa(4k+2)\frac{2}{\pi} \frac{1}{2} \frac{\e_k}{L} \ 
= \ \frac{\gamma}{\pi L^2}.
\end{array}$$ 

(ii) Assume that $\eta_h|_{[0,1]}$ is continuous at $t = 1$. 
Then, there exists $t_0 \in (0, 1)$ such that $\tilde{d}_\kappa(h_t, h) <  a$ for any $t \in (t_0, 1]$. 
Since $\frac{2k}{2k+1} \to 1$ $(k \to \infty)$, there exists $k \in \IN$ such that $t := \frac{2k}{2k+1} > t_0$. 
It follows that 
\begin{itemize} 
\item[(c)] $t \in (t_0, 1)$, \ so \ $\tilde{d}_\kappa(h_t, h) <  a$ \ and 
\item[(d)] $t(2k+1) = 2k$ \ and \ $h_t(2kz) = th(\frac{1}{t}2kz) = th((2k+1)z) = t(2k+1)z = 2kz$. 
\end{itemize} 
Hence we have \ 
$$\tilde{d}_\kappa(h_t, h) \geq \tilde{d}_\kappa(h_t(2kz), h(2kz)) = \tilde{d}_\kappa(2kz, h(2kz)) \geq a,$$ 
which contradicts (c). 
\end{proof} 


\end{document}